\documentclass[11pt,a4paper,reqno]{amsart}
\usepackage[english]{babel}
\usepackage[applemac]{inputenc}
\usepackage[T1]{fontenc}
\usepackage{palatino}
\usepackage{verbatim}
\usepackage{amsmath}
\usepackage{amssymb}
\usepackage{amsthm}
\usepackage{amsfonts}
\usepackage{graphicx}
\usepackage{mathtools}
\usepackage{enumitem}

\usepackage[colorlinks = true, citecolor = black]{hyperref}
\author{Vasileios Chousionis, Katrin F\"assler, and Tuomas Orponen}
\title[Singular integrals on intrinsic graphs]{Boundedness of singular integrals on $C^{1,\alpha}$ intrinsic graphs in the Heisenberg group}
\address{Department of Mathematics, University of Connecticut, USA}
\email{vasileios.chousionis@uconn.edu}
\address{Department of Mathematics, University of Fribourg, Switzerland}
\email{katrin.faessler@unifr.ch}
\address{Department of Mathematics and Statistics, University of Helsinki, Finland}
\email{tuomas.orponen@helsinki.fi}
\thanks{V.C. is supported by  the Simons Foundation via the project `Analysis and dynamics in Carnot groups', Collaboration grant no.\  521845. K.F. is supported by  Swiss National Science Foundation via the project `Intrinsic rectifiability and mapping theory on the Heisenberg group', grant no.\ 161299. T.O. is supported by the Academy of Finland via the project `Restricted families of projections and connections to Kakeya type problems', grant no.\ 274512.}

\newcommand{\Per}{\textup{Per}}
\newcommand{\R}{\mathbb{R}}
\newcommand{\V}{\mathbb{V}}
\newcommand{\He}{\mathbb{H}}
\newcommand{\N}{\mathbb{N}}

\newcommand{\W}{\mathbb{W}}
\newcommand{\Z}{\mathbb{Z}}

\newcommand{\calL}{\mathcal{L}}
\newcommand{\calD}{\mathcal{D}}
\newcommand{\calH}{\mathcal{H}}

\newcommand{\calB}{\mathcal{B}}

\newcommand{\calC}{\mathcal{C}}

\newcommand{\calS}{\mathcal{S}}

\newcommand{\calR}{\mathcal{R}}

\newcommand{\spt}{\operatorname{spt}}

\newcommand{\calK}{\mathcal{K}}

\newcommand{\ra}{\rightarrow}
\newcommand{\ve}{\varepsilon}

\newcommand{\diam}{\operatorname{diam}}

\newcommand{\dist}{\operatorname{dist}}

\newcommand{\sgn}{\operatorname{sgn}}
\newcommand{\Lip}{\operatorname{Lip}}
\newcommand{\stm}{\setminus}

\numberwithin{equation}{section}

\numberwithin{equation}{section}

\newtheorem{thm}{Theorem}[section]

\newtheorem{lemma}[thm]{Lemma}
\newtheorem{cor}[thm]{Corollary}
\theoremstyle{definition}

\newtheorem{proposition}[thm]{Proposition}
\newtheorem{conjecture}[thm]{Conjecture}
\newtheorem{ex}[thm]{Example}
\theoremstyle{definition}
\newtheorem{claim}[thm]{Claim}
\newtheorem{df}[thm]{Definition}
\theoremstyle{definition}
\newtheorem{definition}[thm]{Definition}
\theoremstyle{definition}
\newtheorem{rem}[thm]{Remark}
\newtheorem{remark}[thm]{Remark}

\newcommand{\nref}[1]{(\hyperref[#1]{#1})}

\begin{document}

\begin{abstract} We study singular integral operators induced by $3$-dimensional Calder\'on-Zygmund kernels in the Heisenberg group. We show that if such an operator is $L^{2}$ bounded on vertical planes, with uniform constants, then it is also $L^{2}$ bounded on all intrinsic graphs of compactly supported $C^{1,\alpha}$ functions over vertical planes.

In particular, the result applies to the operator $\calR$ induced by the kernel
\begin{displaymath} \mathcal{K}(z) = \nabla_{\He} \| z \|^{-2}, \qquad z \in \He \setminus \{\mathbf{0}\}, \end{displaymath}
the horizontal gradient of the fundamental solution of the sub-Laplacian. The $L^{2}$ boundedness of $\calR$ is connected with the question of removability for Lipschitz harmonic functions. As a corollary of our result, we infer that the intrinsic graphs mentioned above are non-removable. Apart from subsets of vertical planes, these are the first known examples of non-removable sets with positive and locally finite $3$-dimensional measure.
\end{abstract}

\maketitle

\tableofcontents

\section{Introduction}

The purpose of this paper is to study the boundedness of certain $3$-dimensional singular integrals on intrinsic graphs in the first Heisenberg group $\He$, a $3$-dimensional manifold with a $4$-dimensional metric structure. All the formal definitions will be deferred to Section \ref{s:Defs}, so this introduction will be brief, informal and not entirely rigorous.

We study singular integral operators (SIOs) of convolution type. In $\He$, this refers to objects of the following form:
\begin{equation}\label{SIOIntro} T_{\mu}f(p) = \int K(q^{-1} \cdot p)f(q) \, d\mu(q), \end{equation}
where $K \colon \He \setminus \{\mathbf{0}\} \to \R^{d}$ is a \emph{kernel}, and $\mu$ is a locally finite Borel measure. Specifically, we are interested in the $L^{2}$ boundedness of the operator $f \mapsto T_{\calH}f$ on certain $3$-regular surfaces $\Gamma \subset \He$, where $\calH = \calH^{3}|_{\Gamma}$ is $3$-dimensional Hausdorff measure restricted to $\Gamma$. The relevant surfaces $\Gamma$ are the \emph{intrinsic Lipschitz graphs}, introduced by Franchi, Serapioni and Serra Cassano \cite{MR2287539} in 2006. These are the Heisenberg counterparts of (co-dimension $1$) Lipschitz graphs in $\R^{d}$. In the Euclidean environment, the boundedness of SIOs on Lipschitz graphs, and beyond, is a classical topic, developed by Calder\'on \cite{Calderon}, Coifman-McIntosh-Meyer \cite{CMM}, David \cite{MR956767}, David-Semmes \cite{DS1}, and many others.

If $\Gamma$ is a \emph{vertical plane} $\W$ (a plane in $\He$ containing the vertical axis), then the boundedness of $T_{\calH}$ on $L^{2}(\calH)$ is essentially a Euclidean problem. In fact, as long as $p,q \in \W$, the group operation $p \cdot q$ behaves like addition in $\R^{2}$. Also, $\calH = \calH^{3}|_{\W}$ is simply a constant multiple of $2$-dimensional Lebesgue measure. So, $T_{\calH}$ can be identified with a convolution type SIO in $\R^{2}$.\footnote{One should keep in mind, however, that if $K$ is a Calder\'on-Zygmund kernel in $\He$, then the restriction of $K$ to $\W$ satisfies the standard growth and H\"older continuity estimates with respect to a non-Euclidean metric on $\W$.} The $L^{2}$ boundedness question for such operators is classical, see Stein's book \cite{stein1993harmonic}, and for instance Fourier-analytic tools are applicable.

The main result of the paper, see Theorem \ref{mainIntro} below, asserts that solving the Euclidean problem automatically yields information on the non-Euclidean problem. Before making that statement more rigorous, however, we ask: what are the natural SIOs in $\He$, in the context of the $3$-dimensional surfaces $\Gamma$? In $\R^{d}$, a prototypical singular integral is the $(d - 1)$-dimensional \emph{Riesz transform}, whose kernel is the gradient of the fundamental solution of the Laplacian,
\begin{displaymath} \calK_{\R^{d}}(x) = \nabla |x|^{-(d - 2)}. \end{displaymath}
The boundedness of the associated singular integral operator $\calR_{\R^{d}}$ is connected with the problem of removability for Lipschitz harmonic functions. A closed set $E \subset \R^{d}$ is \emph{removable for Lipschitz harmonic functions}, or just \emph{removable}, if whenever $D \supset E$ is open, and $f \colon D \to \R$ is Lipschitz and harmonic in $D \setminus E$, then $f$ is harmonic in $D$. In brief, the connection between $\calR_{\R^{d}}$ and removability is the following: if $\calR_{\R^{d}}$ is bounded on $\Gamma$ for some closed $(d - 1)$-regular set $\Gamma$, then $\Gamma$, or positive measure closed subsets of $\Gamma$, are \textbf{not} removable for Lipschitz harmonic functions, see Theorem 4.4 in \cite{MR1372240}.  The importance of the Riesz transform in the study of removability is highlighted in the seminal papers by David and Mattila \cite{dm}, and Nazarov, Tolsa and Volberg \cite{ntov, ntov2}. Using, among other things, techniques from non-homogeneous harmonic analysis, they characterise removable sets as  the purely $(d-1)$-unrectifiable sets in $\R^d$, that is, the sets which intersect every $\mathcal{C}^1$ hypersurface in a set of vanishing $(d-1)$-dimensional Hausdorff measure.

In $\He$, the counterparts of harmonic functions are solutions to the sub-Laplace equation $\Delta_{\He} u = 0$, see Section \ref{s:heisenbergGroup}, or \cite{BLU}. With this notion of harmonicity, the problem of removability in $\He$ makes sense, and has been studied in \cite{CM, MR3347479}. Also, as in $\R^{d}$, removability is connected with the boundedness of a certain singular integral $\calR_{\He}$, now with kernel
\begin{displaymath} \calK_{\He}(z) = \nabla_{\He} \|z\|^{-2}, \end{displaymath} where $\|\cdot\|$ denotes the Kor\'{a}nyi distance.
In contrast to the Euclidean case, this kernel is \textbf{not} antisymmetric in the sense $\calK_{\He}(z) = -\calK_{\He}(z^{-1})$. Nevertheless, it is known that the associated SIO $\calR_{\He}$ is $L^{2}$ bounded on vertical planes, see
Remark 3.15
in \cite{CM}.
This is due to the fact that $\calK_{\He}$ is \emph{horizontally antisymmetric}: $\calK_{\He}(x,y,t) = -\calK_{\He}(-x,-y,t)$ for $(x,y,t) \in \He$. On vertical planes $\W$, this amount of antisymmetry suffices to guarantee boundedness on $L^{2}$ by classical results, see for instance Theorem 4 on p. 623 in \cite{stein1993harmonic}.


We now introduce the main theorem. We propose in Conjecture \ref{mainC} that convolution type SIOs with Calder\'on-Zygmund kernels which are uniformly $L^2$-bounded on vertical planes, are also bounded on intrinsic Lipschitz graphs $\Gamma \subset \He$. This would prove that such sets $\Gamma$ are non-removable -- a fact which, before the current paper, was only known for the vertical planes $\W$.
In this paper, we verify Conjecture \ref{mainC} for the intrinsic graphs of compactly supported intrinsically $C^{1,\alpha}(\W)$-functions, defined on vertical planes $\W \subset \He$, see Definitions \ref{d:intrinsicGraph} and \ref{d:C1,alpha,W}. This class contains all compactly supported Euclidean $\calC^{1,\alpha}$-functions, with the identification $\W \cong \R^{2}$, see Remark \ref{e:euclideanHolder}.
\begin{thm}\label{mainIntro} Let $\alpha > 0$, and assume that $\phi \in C^{1,\alpha}(\W)$ has compact support. Then, any convolution type SIO with a Calder\'on-Zygmund kernel which is uniformly $L^2$-bounded on vertical planes, is bounded on $L^{2}(\mu)$ for any $3$-Ahlfors-David regular measure $\mu$ supported on the intrinsic graph $\Gamma$ of $\phi$. In particular, this is true for the $3$-dimensional Hausdorff measure on $\Gamma$.
\end{thm}
In particular, the result applies to the operator $\calR_{\He}$, as its $L^{2}$-boundedness on vertical planes is known. The formal connection between the boundedness of the singular integral $\calR_{\He}$, and removability, is explained in the following result:
\begin{thm}\label{mainRem} Assume that $\mu$ is a non-trivial positive Radon measure on $\He$, satisfying the growth condition $\mu(B(p,r)) \leq Cr^{3}$ for $p \in \He$ and $r > 0$, and such that the support $\spt \mu$ has locally finite $3$-dimensional Hausdorff measure. If $\calR_{\He}$ is bounded on $L^{2}(\mu)$, then $\spt \mu$ is not removable for Lipschitz harmonic functions.
\end{thm}
We prove Theorem \ref{mainRem} in Section \ref{sec:rem}; the argument is nearly the same as the one used by Mattila and Paramonov \cite{MR1372240} in the Euclidean case. There are a few subtle differences, however, so we provide all the details. The proof also requires an auxiliary result of some independent interest, on slicing a set in $\He$ by horizontal lines, see Lemma \ref{l:preImage}.

With Theorems \ref{mainIntro} and \ref{mainRem} in hand, the following corollaries are rather immediate:
\begin{cor}\label{cor1} Let $\alpha > 0$. Assume that $\phi \in C^{1,\alpha}(\W)$ is compactly supported. If $E$ is a closed subset of the intrinsic graph of $\phi$ with positive $3$-dimensional Hausdorff measure, then $E$ is not removable.
\end{cor}
\begin{cor}\label{cor2} Let $\alpha > 0$. Assume that $\Omega \subset \R^{2} \cong \W$ is open, and $\phi$ is Euclidean $\calC^{1,\alpha}$ on $\Omega$. If $E$ is a closed subset of the intrinsic graph of $\phi$ over $\W$ with positive $3$-dimensional Hausdorff measure, then $E$ is not removable.
\end{cor}

The structure of the paper is the following. In Section \ref{s:Defs}, we introduce all the relevant concepts, from singular integrals to (intrinsic) $C^{1,\alpha}$ functions, and prove some simple lemmas. In Section \ref{s:mainProof}, we prove Theorem \ref{mainIntro}. In Section \ref{s:regularity}, we study how well the (intrinsic) graphs of $C^{1,\alpha}$ functions are approximated by vertical planes; this analysis is required in the proof of Theorem \ref{mainIntro}. Finally, in Section \ref{sec:rem}, we study the connection with the removability problem, and prove Theorem \ref{mainRem} and Corollaries \ref{cor1} and \ref{cor2}.

\medskip

\textbf{Acknowledgements.} We thank Lingxiao Zhang for pointing
out a mistake in the proof of an earlier version of Proposition \ref{p:improvedRegularity}.

\section{Definitions and preliminaries}\label{s:Defs}

\subsection{The Heisenberg group and general notation}\label{s:heisenbergGroup}
The \emph{first Heisenberg group} $\He$ is $\mathbb{R}^3$ endowed with the group law
\begin{equation}\label{eq:group_law} z_{1} \cdot z_{2} = (x_{1} + x_{2},y_{1} + y_{2}, t_{1} + t_{2} + \tfrac{1}{2}[x_{1}y_{2} - x_{2}y_{1}]), \end{equation}
for $z_{i} = (x_{i},y_{i},t_{i}) \in \He$. The neutral element in these coordinates is given by $\mathbf{0}=(0,0,0)$ and the inverse of $p=(x,y,t)$ is denoted by $p^{-1}$ and given by $(-x,-y,-t)$. The \emph{Kor\'{a}nyi distance} is defined as
\begin{equation}\label{d:theMetric}
d(z_1,z_2):=\|(z_2)^{-1}\cdot z_1\|,\quad z_1,z_2\in \He,
\end{equation}
where
\begin{equation}\label{d:koranyiNorm}
\|(x,y,t)\|:= \sqrt[4]{(x^2+y^2)^2+16 t^2},\;\quad\text{for }  (x,y,t)\in \He^1.
\end{equation}
A frame for the left invariant vector fields is given by
\begin{displaymath}
X= \partial_x - \tfrac{y}{2}\partial_t,\quad Y= \partial_y+\tfrac{x}{2}\partial_t,\quad\text{and}\quad T=\partial_t.
\end{displaymath}
The \emph{horizontal gradient} of a function $u: \Omega \to \mathbb{R}$ on an open set $\Omega \subseteq \mathbb{H}$ is
\begin{displaymath}
\nabla_{\mathbb{H}} u = (Xu)X + (Yu)Y.
\end{displaymath}
and the \emph{sub-Laplacian} of $u$ is
\begin{equation}\label{eq:subLapl} \Delta_{\He} u = X^{2}u + Y^{2}u. \end{equation}
We consider the horizontal gradient as a mapping with values in $\mathbb{R}^2$ and write $\nabla_{\mathbb{H}} u =(Xu,Yu)$.
For a thorough introduction to the Heisenberg group, we refer the reader to
Chapter 2 of the monograph
\cite{MR2312336}.

\subsubsection{Notation} We usually denote points of $\He$ by $z,p$ or $q$; in coordinates, we often write $z = (x,y,t)$ with $x,y,t \in \R$. Points on vertical subgroups $\W \subset \He$ (see Section \ref{s:intrinsicLipschitz}) are typically denoted by $w$.

Unless otherwise specified, all metric concepts in the paper, such as the diameter and distance of sets, are defined using the metric $d$ given in \eqref{d:theMetric}. The notation $|\cdot |$ refers to Euclidean norm, and $\|\cdot\|$ refers to the quantity defined in \eqref{d:koranyiNorm}. A closed ball in $(\He,d)$ of radius $r > 0$ and centre $z \in \He$ is denoted by $B(z,r)$.

For $A,B > 0$, we use the notation $A \lesssim_{h} B$ to signify that there exists a constant $C \geq 1$, depending only on the parameter "$h$", such that $A \leq CB$. If no "$h$" is specified, the constant $C$ is absolute. We abbreviate the two-sided inequality $A \lesssim_{h} B \lesssim_{h} A$ by $A \sim_{h} B$.

The notation $\calH^{s}$ stands for the $s$-dimensional Hausdorff measure (with respect to the metric $d$), and Lebesgue measure on $\R^{2}$ is denoted by $\calL^{2}$; this notation is also used to denote Lebesgue measure on the subgroups $\W$ under the identification $\W \cong \R^{2}$.

\subsection{Kernels and singular integral operators in $\He$}\label{ss:KernelsAndSIO} An aim of the paper is to study the $L^{2}$ boundedness of singular integral operators (SIOs) on fairly smooth $3$-Ahlfors-David regular surfaces in $\He$ (see Section \ref{s:C1alphaFunctions} for a more precise description of our surfaces). But what are these SIOs -- and what are their kernels? In this paper, a \emph{kernel} is any continuous function $K \colon \He \setminus \{\mathbf{0}\} \to \R^{d}$. Motivated by similar considerations in Euclidean spaces, it seems reasonable to impose the following growth and H\"older continuity estimates:
\begin{equation}\label{form0} |K(z)| \lesssim \frac{1}{\|z\|^{3}} \quad \text{and} \quad |K(z_{1}) - K(z_{2})| \lesssim \frac{\|z_{2}^{-1} \cdot z_{1}\|^{\beta}}{\|z_{1}\|^{3 + \beta}}, \end{equation}
for some $\beta \in (0,1]$, and for all $z \in \He \setminus \{\mathbf{0}\}$ and $z_{1},z_{2}\in \He \setminus \{\mathbf{0}\}$ with $d(z_{1},z_{2}) \leq \|z_{1}\|/2$. We call such kernels \emph{$3$-dimensional Calder\'on-Zygmund (CZ) kernels in $\He$}. The conditions above, and the lemma below, imply that our terminology is consistent with standard terminology, see for instance p. 293 in Stein's book \cite{stein1993harmonic}.
\begin{lemma}\label{l:standardHolder} Assume that a kernel $K \colon \He \setminus \{\mathbf{0}\} \to \R^{d}$ satisfies the second (H\"older continuity) estimate in \eqref{form0} for some $\beta > 0$. Then,
\begin{equation}\label{form42} |K(q^{-1} \cdot p_{1}) - K(q^{-1} \cdot p_{2})| + |K(p_{1}^{-1} \cdot q) - K(p_{2}^{-1} \cdot q)| \lesssim \frac{\|p_{2}^{-1} \cdot p_{1}\|^{\beta/2}}{\|q^{-1}\cdot p_{1}\|^{3 + \beta/2}} \end{equation}
for $q\in \He$, $p_{1},p_{2} \in \He \setminus\{q\}$ with $d(p_{1},p_{2}) \leq d(p_{1},q)/2$.
\end{lemma}
\begin{proof} Write $z_{1} := q^{-1} \cdot p_{1}$ and $z_{2} := q^{-1} \cdot p_{2}$. Then $d(z_{1},z_{2}) \leq \|z_{1}\|/2$ by left-invariance of $d$, so the first summand in \eqref{form42} has the correct bound by \eqref{form0}, even with $\beta/2$ replaced by $\beta$. Hence, to find a bound for the second summand, we only need to prove that
\begin{displaymath} |K(z_{1}^{-1}) - K(z_{2}^{-1})| \lesssim \frac{\|z_{2}^{-1} \cdot z_{1}\|^{\beta/2}}{\|z_{1}\|^{3 + \beta/2}}. \end{displaymath}
We may moreover assume that $d(z_{1},z_{2}) \leq \|z_{1}\|/C$ for a suitable large constant $C \geq 1$. We would like to apply \eqref{form0} as follows,
\begin{equation}\label{form43} |K(z_{1}^{-1}) - K(z_{2}^{-1})| \lesssim \frac{\|z_{2} \cdot z_{1}^{-1}\|^{\beta}}{\|z_{1}\|^{3 + \beta}}, \end{equation}
but we first need to make sure that $d(z_{1}^{-1},z_{2}^{-1}) \leq \|z_{1}\|/2$. Write $z_{1} = (x_{1},y_{1},t_{1})$ and $z_{2} = (x_{2},y_{2},t_{2})$, and observe that
\begin{align} d(z_{1}^{-1},z_{2}^{-1}) = \|z_{2} \cdot z_{1}^{-1}\| & = \|(x_{2} - x_{1},y_{2} - y_{1},t_{2} - t_{1} - \tfrac{1}{2}[x_{2}y_{1} - y_{2}x_{1}]\| \notag\\
& \lesssim \|(x_{2} - x_{1},y_{2} - y_{1},t_{2} - t_{1} + \tfrac{1}{2}[x_{2}y_{1} - y_{2}x_{1}]\| \notag\\
&\qquad + \sqrt{|x_{2}y_{1} - x_{1}y_{2}|} \notag\\
& = d(z_{1},z_{2}) + \sqrt{|(x_{2} - x_{1})y_{1} - (y_{2} - y_{1})x_{1}|} \notag\\
&\label{form44} \lesssim d(z_{1},z_{2}) + \sqrt{d(z_{1},z_{2})}\sqrt{\|z_{1}\|}. \end{align}
It follows from \eqref{form44} that $d(z_{1}^{-1},z_{2}^{-1}) \leq \|z_{1}\|/2$, if the constant $C$ was chosen large enough. Hence, the estimate \eqref{form43} is legitimate, and we may further use \eqref{form44} obtain
\begin{displaymath} |K(z_{1}^{-1}) - K(z_{2}^{-1})| \lesssim \frac{\|z_{2}^{-1} \cdot z_{1}\|^{\beta}}{\|z_{1}\|^{3 + \beta}} + \frac{\|z_{2}^{-1} \cdot z_{1}\|^{\beta/2}\|z_{1}\|^{\beta/2}}{\|z_{1}\|^{3 + \beta}} \lesssim \frac{\|z_{2}^{-1} \cdot z_{1}\|^{\beta/2}}{\|z_{1}\|^{3 + \beta/2}}, \end{displaymath}
as claimed. \end{proof}

We now recall some basic notions about SIOs. Fix a $3$-dimensional Calder\'on-Zygmund kernel $K \colon \He \setminus \{\mathbf{0}\} \to \R^{d}$, and a complex Radon measure $\nu$. For $\epsilon > 0$, we define
\begin{displaymath} T_{\epsilon}\nu(p) := \int_{\|q^{-1} \cdot p\| > \epsilon} K(q^{-1} \cdot p) \, d\nu (q), \qquad p \in \He, \end{displaymath}
whenever the integral on the right hand side is absolutely convergent; this is, for instance, the case if $\nu$ has finite total variation. Next, fix a positive Radon measure $\mu$ on $\He$ satisfying the growth condition
\begin{equation}\label{muGrowth} \mu(B(p,r)) \leq Cr^{3}, \qquad p \in \He, \: r > 0, \end{equation}
where $C \geq 1$ is a constant. Given a complex function $f \in L^{2}(\mu)$ and $\epsilon > 0$, we define
\begin{displaymath} T_{\mu,\epsilon}f(p) := T_{\epsilon}(f \, d\mu)(p), \qquad p \in \He. \end{displaymath}
It easily follows from the growth conditions on $K$ and $\mu$, and Cauchy-Schwarz inequality, that the expression on the right makes sense for all $\epsilon > 0$.
\begin{definition}\label{d:SIOBoundedness} Given a $3$-dimensional CZ kernel $K$, and a measure $\mu$ satisfying \eqref{muGrowth}, we say that \emph{the SIO $T$ associated to $K$ is bounded on $L^{2}(\mu)$}, if the operators
\begin{displaymath} f \mapsto T_{\mu,\epsilon}f \end{displaymath}
are bounded on $L^{2}(\mu)$ with constants independent of $\epsilon > 0$.
\end{definition}

For the rest of the paper, we are mainly concerned with measures $\mu$ satisfying the $2$-sided inequality $c r^{3} \leq \mu(B(p,r)) \leq C r^{3}$ for all $p \in \spt \mu$ and $0 < r \leq \diam(\spt \mu)$, and for some fixed constants $0 < c < C < \infty$. Such measures are called \emph{$3$-Ahlfors-David regular}, or \emph{$3$-ADR} in short.

\begin{rem}
\label{r:adjoint} Given a $3$-dimensional $CZ$ kernel $K$,  the kernel $K^{\ast}$, defined by $K^{\ast}(p) := K(p^{-1})$, is the kernel of the formal adjoint $T_{\mu,\epsilon}^{\ast}$ of $T_{\mu,\epsilon}$ since
\begin{align*} \int (T_{\mu,\epsilon}f) g \, d\mu & = \int \left( \int_{\|q^{-1} \cdot p\| > \epsilon} K(q^{-1} \cdot p)f(q) \, d\mu (q) \right) g(p) \, d\mu (p)\\
& = \int \left( \int_{\|p^{-1} \cdot q\| > \epsilon} K^{\ast}(p^{-1} \cdot q) g(p) \, d\mu (p) \right) f(q) \, d\mu (q) = \int (T^{\ast}_{\mu,\epsilon}g)f \, d\mu. \end{align*}
It is easy to check that $K^{\ast}$  satisfies the growth condition in \eqref{form0}. Moreover, $K^{\ast}$ satisfies the H\"older continuity requirement in \eqref{form0} with exponent $\beta/2$; this is a corollary of Lemma \ref{l:standardHolder}. 
\end{rem}

\subsubsection{Two examples}\label{s:exampleKernels} In this short section, we give two examples of concrete $3$-dimensional CZ kernels.

\begin{ex}[The $3$-dimensional $\He$-Riesz kernel] Consider the kernel
\begin{equation}\label{rieszKernel} \mathcal{K}(z) = \nabla_{\He} \|z\|^{-2}, \qquad z \in \He \setminus \{0\}. \end{equation}
Note that $\|\cdot\|^{-2}$ agrees (up to a multiplicative constant) with the fundamental solution of the sub-Laplacian $\Delta_{\He}$, as proved by Folland \cite{MR0315267}, see also \cite[Example 5.4.7]{BLU}.
We call $\mathcal{K}$ the \emph{$3$-dimensional $\He$-Riesz kernel}; it gives rise to a SIO $\calR$, which we call the \emph{$3$-dimensional $\He$-Riesz transform}. Studying the $L^{2}$-boundedness of $\calR$ on subsets of $\He$ is connected with the removability of these sets for Lipschitz harmonic functions on $\He$, see Theorem \ref{mainRem2} for the precise statement.

The neat formula \eqref{rieszKernel} can be expanded to the following rather unwieldy expression:
\begin{equation}\label{KExplicit} \mathcal{K}(x,y,t) = \left(\frac{-2x|(x,y)|^{2} + 8yt}{\|(x,y,t)\|^{6}}, \frac{-2y|(x,y)|^{2} - 8xt}{\|(x,y,t)\|^{6}} \right) =: (\mathcal{K}^{1}(x,y,t),\mathcal{K}^{2}(x,y,t)). \end{equation}
From the formula above, one sees that $\mathcal{K}$ is not antisymmetric in the usual sense $\mathcal{K}(z^{-1}) = -\mathcal{K}(z)$; for instance, $\mathcal{K}^{1}(0,1,1) = \mathcal{K}^{1}(0,-1,-1)$. However, both components of $\mathcal{K}$ are \emph{horizontally antisymmetric}, as in the definition below. \end{ex}

\begin{definition}[Horizontal antisymmetry] A kernel $K \colon \He \setminus \{\mathbf{0}\} \to \R$ is called \emph{horizontally antisymmetric}, if
\begin{displaymath} K(x,y,t) = -K(-x,-y,t), \qquad (x,y,t) \in \He \setminus \{\mathbf{0}\}. \end{displaymath}
\end{definition}
It is clear from the formula \eqref{KExplicit} that $\mathcal{K}$ is horizontally antisymmetric.

\begin{ex}[The $3$-dimensional quasi $\He$-Riesz kernel]\label{s:qHRieszKernel}

Consider
\begin{equation}
\label{quasiriesz}
\Omega (x,y,t):= \left( \frac{x}{\|(x,y,t)\|^4}, \frac{y}{\|(x,y,t)\|^4}, \frac{t}{\|(x,y,t)\|^5}\right), \qquad (x,y,t) \neq (0,0,0).
\end{equation}
It is easy to see that ${\Omega}$ is a $3$-dimensional CZ kernel which is \emph{antisymmetric} in the sense that $\Omega(p)=-\Omega(p^{-1})$ for all $p\in \He \setminus \{0\}$. We will call $\Omega$ the {\em $3$-dimensional quasi $\He$-Riesz kernel}; it defines the  {\em $3$-dimensional quasi $\He$-Riesz transform} $\mathcal{Q}$.  The kernel $\Omega$, which resembles in form the Euclidean Riesz kernels, was introduced in \cite{Chousionis2011}. It was proved there that if $\mu$ is a $3$-ADR measure and $\mathcal{Q}$ is bounded in $L^2(\mu)$ then $\spt \mu$ can be approximated at $\mu$ almost every point and at arbitrary small scales by homogeneous subgroups. It is unknown if the $\He$-Riesz transform has the same property. \end{ex}

\subsubsection{Cancellation conditions}\label{s:cancel} Fix a $3$-dimensional CZ kernel $K \colon \He \setminus \{0\} \to \R$. Without additional assumptions, the SIO $T$ associated with $K$ is generally not bounded on $L^{2}(\mu)$, even when $\mu$ is nice, such as the $3$-dimensional Hausdorff measure on a vertical plane $\W$ (see Section \ref{s:intrinsicLipschitz} for a definition of these planes), which is a constant multiple of Lebesgue measure on $\W$. So, for positive results, one needs to impose \emph{cancellation conditions}. The \emph{horizontal antisymmetry} or \emph{antisymmetry} would be such conditions, but neither of them holds both for the $\He$-Riesz kernel and the quasi $\He$-Riesz kernel simultaneously.  Here is a more general cancellation condition, which encompasses antisymmetric and horizontally antisymmetric kernels:
\begin{df}[AB]\label{abc}
A kernel $K: \He \setminus \{\mathbf{0}\} \ra \R$ satisfies the {\em annular boundedness} condition (AB for short) if the following holds. For every every $\|\cdot\|$-radial $\calC^{\infty}$ function $\psi : \He \ra \R$ satisfying $\chi_{B(0,1/2)} \leq \psi \leq \chi_{B(0,2)}$, there exists a constant $A_{\psi} \geq 1$ such that
\begin{equation}
\label{vabann} \left|\int_{\W} [\psi^{R}(w) - \psi^{r}(w)]K(w) \, d\calL^{2}(w)\right| \leq A_{\psi} \end{equation}
for all $0 < r < R < \infty$, and for all vertical planes $\W$. Above,
\begin{displaymath} \psi^{r}(z) := (\psi \circ \delta_{1/r})(z),
\end{displaymath}
where $\delta_{r}$ is the ($\|\cdot\|$-homogeneous) dilatation $\delta_{r}(x,y,t) = (rx,ry,r^{2}t)$ for $(x,y,t) \in \He$.
\end{df}
It turns out that the AB condition for $3$-dimensional CZ kernels $K$ is equivalent to the following condition:

\begin{definition}[UBVP]
\label{univert} Given a kernel $K \colon \He \setminus \{\mathbf{0}\} \to \R$ with $\|K(z)\| \lesssim \|z\|^{-3}$, we say that it is \emph{uniformly $L^2$ bounded on vertical planes} (UBVP in short), if  the SIO  associated to $K$  is bounded on $L^{2}(\calH^{3}|_{\W})$ for every vertical plane $\W$ (in the sense of Definition \ref{d:SIOBoundedness}), with constants independent of $\W$. \end{definition}

The measure $\calH^{3}|_{\W}$ is $3$-ADR, so it makes sense to discuss boundedness of $T$ on $L^{2}(\calH^{3}|_{\W})$. The following lemma is an analog of \cite[Proposition 2, p. 291]{stein1993harmonic}.

\begin{lemma} Assume that a kernel $K \colon \He \setminus \{\mathbf{0}\} \to \R$ with $\|K(z)\| \lesssim \|z\|^{-3}$ satisfies the UBVP condition. Then $K$ also satisfies the AB condition. \end{lemma}
\begin{proof} Fix a vertical plane $\W \subset \He$. The group operation "$\cdot$" restricted to $\W$ coincides with usual (Euclidean) addition in the plane $\W$: if $v,w \in \W$, then $v^{-1} \cdot w = w - v$. Also, $\calH^{3}|_{\W} = c \cdot \calL^{2}$ for some positive constant $c$. Hence, for $w \in \W$,
\begin{displaymath} T_{\calH^{3}|_{\W},\epsilon}f(w) = c\int_{\W} K(w - v)B_{\epsilon}(w - v) f(v)\, d\mathcal{L}^2(v) = c(KB_{\epsilon}) \ast f(w),\quad f \in \calC^{\infty}_{0}(\W), \end{displaymath}
where $B_{\epsilon}$ is the indicator function of $\W \setminus B(0,\epsilon)$, the notation "$\ast$" means Euclidean convolution, and $\calC^{\infty}_{0}(\W)$ stands for smooth and compactly supported functions on $\W$. Since $T_{\calH^{3}|_{\W},\epsilon}$ is $L^{2}$ bounded on $\W \cong \R^{2}$, it follows that the Fourier transform of $KB_{\epsilon}$ is a bounded function on $\W$, independently of $\epsilon$:
\begin{displaymath} |\widehat{KB_{\epsilon}}(\xi)| \leq A, \qquad \xi \in \W, \: \epsilon > 0. \end{displaymath}
For the proof see e.g. \cite[2.5.10]{Grafakos}. Now, fix a function $\psi \colon \He \to \R$ as in Definition \ref{abc}. Fix also $0 < r < R < \infty$, and a vertical plane $\W$. Note that $\psi^{R} - \psi^{r}$ vanishes in the ball $B(0,r/2)$. Hence, if $0 < \epsilon < r/2$, we have $\psi^{R}(w) - \psi^{r}(w) = [\psi^{R}(w) - \psi^{r}(w)]B_{\epsilon}(w)$ for $w \in \W$, and hence
\begin{align*} \left|\int [\psi^{R}(w) - \psi^{r}(w)]K(w) \, d\calL^{2}(w)\right| & = \left|\int [\psi^{R}(w) - \psi^{r}(w)]K(w)B_{\epsilon}(w) \, d\calL^{2}(w)\right|\\
& \lesssim A\int  \left|\widehat{\psi^{R}}(\xi) - \widehat{\psi^{r}}(\xi)\right| \, d\calL^{2}(\xi),  \end{align*}
using Plancherel before passing to the second line. Moreover,
\begin{displaymath} \int |\widehat{\psi^{R}}(\xi)| \, d\calL^{2}(\xi) = \int |\widehat{\psi}(\delta_{R}(\xi))|R^{3} \, d\calL^{2}(\xi) = \int |\widehat{\psi}(\xi)| \, d\calL^{2}(\xi) \lesssim 1, \end{displaymath}
and the same holds with "$r$" in place of "$R$". This completes the proof. \end{proof}

Now, recall the main result, Theorem \ref{mainIntro}. With the terminology above, it states that if a $3$-dimensional CZ kernel satisfies the UBVP, then the associated SIO is bounded on certain $L^{2}$ spaces, which we will define momentarily (see Section \ref{s:C1alphaFunctions}). The strategy of proof is to infer, from the lemma above, that the kernel satisfies the AB condition, and proceed from there. In particular, it remains to prove the following version Theorem \ref{mainIntro}:
\begin{thm}\label{mainABC} Let $\alpha > 0$, and assume that $\phi \in C^{1,\alpha}(\W)$ has compact support. Assume that a $3$-dimensional CZ kernel $K \colon \He \setminus \{\mathbf{0}\} \to \R$ satisfies the AB condition. Then, the associated SIO is bounded on $L^{2}(\mu)$ for any $3$-Ahlfors-David regular measure $\mu$ supported on the intrinsic graph of $\phi$. \end{thm}

As simple corollaries, the $\He$-Riesz transform $\mathcal{R}$ and the quasi $\He$-Riesz transform $\mathcal{Q}$, recall Section \ref{s:exampleKernels}, are $L^2$ bounded on the intrinsic graphs mentioned in Theorem \ref{mainABC}. Since the associated kernels $\calK$ and $\Omega$ are $3$-dimensional CZ kernels, it suffices to verify that they satisfy the AB condition. But this is a consequence of either horizontal antisymmetry (in the case of $\calK$) or antisymmetry (in the case of $\Omega$). In fact, the key cancellation condition \eqref{vabann} even holds in the stronger form
\begin{displaymath} \int [\psi^{R}(w) - \psi^{r}(w)]K(w) \, d\calL^{2}(w) = 0 \end{displaymath}
for all functions $\psi$ as in Definition \ref{abc}, for all $0 < r < R < \infty$, and all vertical planes $\W$.

\subsection{Intrinsic graphs and the boundedness of SIOs}\label{s:C1alphaFunctions} For which $3$-ADR measures $\mu$ are the SIOs associated to $3$-dimensional $CZ$ kernels satisfying the UBVP condition  bounded on $L^{2}(\mu)$? The following seems like a natural conjecture:
\begin{conjecture}\label{mainC} Let $\W \subset \He$ be a vertical subgroup with complementary subgroup $\V$, and let $\phi \colon \W \to \V$ be an intrinsic Lipschitz function (see Definition \ref{d:intrinsicGraph}). If $T$ is a convolution type SIO with a $3$-dimensional CZ kernel which satisfies the UBVP condition, then it is bounded on $L^{2}(\mu)$ for all $3$-ADR measures $\mu$ supported on the intrinsic graph $\Gamma(\phi)$. In particular, this is true for $\mu = \calH^{3}|_{\Gamma(\phi)}$ (since $\calH^{3}|_{\Gamma(\phi)}$ is $3$-ADR by Theorem 3.9 in \cite{MR3511465}).
\end{conjecture}

Recall that the main theorem of the paper, Theorem \ref{mainIntro}, states that the conjecture holds for $\phi \colon \W \to \V$, which are compactly supported and intrinsically $C^{1,\alpha}(\W)$-smooth for some $\alpha \in (0,1]$, see Definition \ref{d:C1,alpha,W}.

\subsubsection{Intrinsic Lipschitz graphs}\label{s:intrinsicLipschitz}
In our terminology, the \emph{vertical subgroups} in $\He$ are all the nontrivial homogeneous normal subgroups of $\He$, except for the center of the group. Recall that \emph{homogeneous subgroups} are subgroups of $\He$ which are preserved under dilations of the Heisenberg group, see \cite{MR3587666}.
With the choice of coordinates as in \eqref{eq:group_law}, the vertical subgroups coincide therefore with the $2$-dimensional subspaces of $\mathbb{R}^3$ that contain the $t$-axis. To every vertical subgroup $\mathbb{W}$ we associate a \emph{complementary horizontal subgroup} $\mathbb{V}$. In our coordinates this is simply the $1$-dimensional subspace in $\mathbb{R}^3$ which is perpendicular to $\mathbb{W}$. Every point $p \in \mathbb{H}$ can be written as $p=p_{\mathbb{W}} \cdot p_{\mathbb{V}}$ with a uniquely determined vertical component $p_{\mathbb{W}}\in \mathbb{W}$ and horizontal component $p_{\mathbb{V}}\in \mathbb{V}$. This gives rise to the \emph{Heisenberg projections}
\begin{displaymath}
\pi_{\mathbb{W}}:\mathbb{H} \to \mathbb{W},\quad \pi_{\mathbb{W}}(p)=p_{\mathbb{W}}
\end{displaymath}
and
\begin{displaymath}
\pi_{\mathbb{V}}:\mathbb{H} \to \mathbb{V},\quad \pi_{\mathbb{V}}(p)=p_{\mathbb{V}}.
\end{displaymath}

\begin{definition}\label{d:intrinsicGraph}
An \emph{intrinsic graph} is a set of the form
\begin{displaymath} \Gamma(\phi) = \{w \cdot \phi(w) : w \in \W\}, \end{displaymath}
where $\W \subset \He$ is a {vertical subgroup} with complementary horizontal subgroup $\V$, and $\phi \colon \W \to \V$ is any function. We often use the notation $\Phi$ for the \emph{graph map} $\Phi(w) = w \cdot \phi(w)$. To define \emph{intrinsic Lipschitz graphs}, fix a parameter $\gamma > 0$, and consider the set (cone)
\begin{displaymath} C_{\gamma} = \{z \in \He: \|\pi_{\W}(z)\| \leq \gamma \|\pi_{\V}(z)\|\}. \end{displaymath}
We say that $\phi$ is an \emph{intrinsic $L$-Lipschitz function}, and $\Gamma(\phi)$ an intrinsic $L$-Lipschitz graph, if
\begin{displaymath}
(z\cdot C_{\gamma})\cap \Gamma(\phi)=\{z\},\quad\text{for }z\in \Gamma(\phi)\text{ and }0<\gamma<\tfrac{1}{L}.
\end{displaymath}
The function $\phi$ is said to be \emph{intrinsic Lipschitz} if it is intrinsic $L$-Lipschitz for some constant $L\geq 0$.
\end{definition}

\begin{remark}\label{r:parametrisation} Every vertical subgroup $\mathbb{W}$ can be parametrised as
\begin{displaymath}
\mathbb{W} = \{(- w_{1} \sin \theta,w_{1} \cos \theta, w_{2}):\; (w_1,w_2)\in \mathbb{R}^2\}
\end{displaymath}
with an angle $\theta \in [0,\pi)$ uniquely determined by $\mathbb{W}$. The complementary horizontal subgroup is then given by
\begin{displaymath}
\mathbb{V} = \{ (v \cos \theta, v \sin \theta,0):\; v \in \mathbb{R}\}.
\end{displaymath}
We often denote points on $\W$ in coordinates by "$(w_{1},w_{2})$", and points on $\V$ by real numbers "$v$". Then, expressions such as $(w_{1},w_{2}) \cdot (w_{1}',w_{2}')$ and $(w_{1},w_{2}) \cdot v$ should be interpreted as elements in $\He$, namely the products of the corresponding elements on $\W$ and $\V$.  \end{remark}

Intrinsic Lipschitz graphs were introduced by Franchi, Serapioni and Serra Cassano in \cite{MR2287539}, motivated by
the study of locally finite perimeter sets and rectifiability in the Heisenberg group \cite{MR1871966}. While intrinsic Lipschitz functions continue to be studied as a class of mappings which are interesting in their own right, they have also recently found a prominent application in \cite{2017arXiv170100620N}. Various properties of intrinsic Lipschitz functions are discussed in detail in
\cite{MR3587666}. For instance, it is known that an intrinsic Lipschitz function has a well-defined \emph{intrinsic gradient} $\nabla^{\phi}\phi \in L^{\infty}(\calH^{3}|_{\W})$, which we will use to perform integration on intrinsic Lipschitz graphs.

\subsubsection{Intrinsic differentiability}

To define the intrinsic gradient, we recall that the notion of intrinsic graph is left invariant. Indeed, given a function $\phi:\mathbb{W}\to \mathbb{V}$ with intrinsic graph $\Gamma(\phi)$, for every $p \in \He$, the set $p \cdot \Gamma(\phi)$ is again the intrinsic graph of a function $\mathbb{W} \to \mathbb{V}$, which we denote by $\phi^{p}$, so that $p \cdot \Gamma(\phi) = \Gamma(\phi^{p})$. For instance if $\mathbb{W}$ is the $(y,t)$-plane, $\mathbb{V}$ the $x$-axis, and $p_{0} = (x_{0},y_{0},t_{0})$, then we can compute explicitly
\begin{equation}\label{eq:translated_function} \phi^{p_{0}}(y,t) = \phi(y - y_{0},t - t_{0} + \tfrac{1}{2}x_{0}y_{0} - yx_{0}) + x_{0}. \end{equation}

We also recall that in our context an \emph{intrinsic linear map} is a function $G:\mathbb{W}\to \mathbb{V}$ whose intrinsic graph is a vertical subgroup.
\begin{definition}
A function $\psi:\mathbb{W} \to \mathbb{V}$ with $\psi(\textbf{0})=\textbf{0}$ is \emph{intrinsically differentiable at $\mathbf{0}$} if
there exists an intrinsic linear map $G:\mathbb{W} \to \mathbb{V}$ such that
\begin{equation}\label{eq:diff_0}
\|(Gw)^{-1} \cdot \psi(w)\| = o(\|w\|),\quad\text{as }w\to \mathbf{0}.
\end{equation}
The map $G$ is called the \emph{intrinsic differential} of $\psi$ at $\textbf{0}$ and denoted by $G= d\psi_{\mathbf{0}}$.

More generally, a function $\phi:\mathbb{W}  \to \mathbb{V}$ is \emph{intrinsically differentiable at a point $w_0\in \mathbb{W}$} if $\psi:=\phi^{(p_0^{-1})}$ is intrinsic differentiable at $\mathbf{0}$ for $p_0:= w_0 \cdot \phi(w_0)$. The  \emph{intrinsic differential of $\phi$ at $w_0$} is given by
\begin{displaymath}
d\phi_{w_0}:= d \psi_{\mathbf{0}}.
\end{displaymath}
\end{definition}
Recall that $\mathbb{V}$ can be identified with $\mathbb{R}$ through our choice of coordinates, see Remark \ref{r:parametrisation}. Under this identification the restriction of the Kor\'{a}nyi distance to $\mathbb{V}$ agrees with the Euclidean distance $| \cdot |$ so that \eqref{eq:diff_0} reads
\begin{equation*}
|\psi(w)-Gw| = o(\|w\|),\quad\text{as }w\to \mathbf{0}.
\end{equation*}
With the parametrisation from Remark \ref{r:parametrisation}, every intrinsic linear map $G: \mathbb{W} \to \mathbb{V}$ has the form $G(w_1,w_2)= c w_1$ for a constant $c\in \mathbb{R}$.

\begin{definition}\label{d:IntrGrad}
Assume that $\phi:\mathbb{W}\to \mathbb{V}$ is intrinsically differentiable at a point $w_0 \in\mathbb{W}$. Then its \emph{intrinsic gradient at $w_0$} is the unique number $\nabla^{\phi}\phi(w_0)$ such that
\begin{displaymath}
d\phi_{w_0} (w_1,w_2) = \nabla^{\phi}\phi(w_0) w_1,\quad\text{for all }(w_1,w_2)\in \mathbb{W}.
\end{displaymath}
\end{definition}

The intrinsic gradient $\nabla^{\phi}\phi(w_0)$ is simply a number determined by the "angle" between $\W$ and the vertical plane $d\phi_{w_{0}}(\W)$.

\subsubsection{$C^{1,\alpha}$-intrinsic Lipschitz functions and graphs}\label{s:C1alphaGraphs} The goal of the paper is to prove that certain SIOs are bounded in $L^2$ on intrinsic graphs $\Gamma(\phi)$, where $\phi$ is a compactly supported function satisfying a (Heisenberg analogue of) $C^{1,\alpha}$-regularity. The most obvious definition of $C^{1,\alpha}$ would be to require the intrinsic gradient $\nabla^{\phi} \phi$ to be locally $\alpha$-H\"older function in the metric space $(\W,d)$, but this condition is not left-invariant: the parametrisation of the left-translated graph $p^{-1} \cdot \Gamma(\phi)$, for $p \in \Gamma(\phi)$, would not necessarily be locally $\alpha$-H\"older continuous with the same exponent $\alpha$, see Example \ref{ex1}. So, instead, we define an "intrinsic" notion of $C^{1,\alpha}$, which is (a) left-invariant in the sense above, and (b) is well-suited for the application we have in mind, and (c) is often easy to verify, see Remark \ref{e:euclideanHolder} below.

\begin{definition}\label{d:C1,alpha,W}
We say that a function $\phi:\mathbb{W} \to \mathbb{V}$ is an \emph{intrinsic $C^{1}(\W)$ function} if $\nabla^{\phi}\phi$ exists at every point $w \in \W$, and is continuous. We further define the subclasses $C^{1,\alpha}(\W)$, $\alpha \in (0,1]$, as follows: $\phi \in C^{1,\alpha}(\W)$, if $\phi \in C^{1}(\W)$, and there exists a constant $H \geq 1$ such that
\begin{align}\label{holder} |\nabla^{\phi^{(p_{0}^{-1})}}\phi^{(p_{0}^{-1})}(w) - \nabla^{\phi^{(p_{0}^{-1})}}\phi^{(p_{0}^{-1})}(\mathbf{0})| \leq H\|w\|^{\alpha}, \end{align}
for all $p_{0} \in \Gamma(\phi)$, and all $w \in \W$. For notational convenience, we also define $C^{1,0}(\W) := C^{1}(\W)$. Intrinsic graphs of $C^{1}$ (or $C^{1,\alpha}$) functions will be called \emph{intrinsic $C^{1}$ (or $C^{1,\alpha}$) graphs}.
\end{definition}

Several remarks are now in order.
\begin{remark}\label{r:ab}
(a) It is well-known, see for instance Proposition 4.4 in \cite{MR3168633} or Lemma 4.6 in \cite{CFO}, that if $\phi \in C^{1}(\W)$ is intrinsic Lipschitz, then $\nabla^{\phi}\phi \in L^{\infty}(\W)$.

(b) Conversely, if $\phi \in C^{1}(\W)$ with $\nabla^{\phi}\phi \in L^{\infty}(\W)$, then $\phi$ is intrinsic Lipschitz. This is well-known and follows from existing results, but it was difficult to find a reference to this particular statement; hence we include the argument in Lemma \ref{l:C1ImpliesLipschitz} below.

\end{remark}

\begin{remark}\label{r:compactSupport} Note that if $\phi \in C^{1}(\W)$ has compact support, then $\nabla^{\phi}\phi \in L^{\infty}(\W)$, and hence $\phi$ is intrinsic Lipschitz by (b) above.

\end{remark}


\begin{remark}  If $\W$ is the $(y,t)$-plane, the condition \eqref{holder} for $p_0=(y_0,t_0)\cdot \phi(y_0,t_0)$ can be written in coordinates as follows:
\begin{equation}\label{form25} |\nabla^{\phi}\phi\big(y + y_{0},t + t_{0} + \phi(y_{0},t_{0})y\big) - \nabla^{\phi}\phi(y_{0},t_{0})| \leq H\|(y,t)\|^{\alpha}. \end{equation}
To see this, apply the representation \eqref{eq:translated_function} with $p_{0}$ replaced by $p_{0}^{-1}$.
\end{remark}

\begin{remark}\label{r:Hregularity}  It is known by Theorem 4.95 in \cite{MR3587666}  that the intrinsic graph of an intrinsic $C^{1}(\W)$ function is an \emph{$H$-regular surface}; in particular, $\phi$ satisfies an area formula, see Section \ref{s:areaFormula} for more details. \end{remark}

The definitions of (intrinsic) $C^{1}(\W)$ and $C^{1,\alpha}(\W)$ are quite different from their standard Euclidean counterparts, which we denote by $\calC^{1}(\R^{2})$ and $\calC^{1,\alpha}(\R^{2})$ (a function belongs to $\calC^{1,\alpha}(\R^{2})$ if its partial derivatives exist and are $\alpha$-H\"older continuous with respect to the Euclidean metric).
However, at least for compactly supported functions, sufficient regularity in the Euclidean sense also implies regularity in the intrinsic Heisenberg sense, as the following remark shows.

\begin{remark}\label{e:euclideanHolder} Assume that $\W$ is the $(y,t)$-plane, and identify $\W$ with $\R^{2}$. Then, any compactly supported $\calC^{1,\alpha}(\R^{2})$-function is in the class $C^{1,\alpha}(\W)$. Indeed, if $\phi \in \calC^{1}(\R^{2})$ then $\nabla^{\phi}\phi$ has the following expression:
\begin{equation}\label{burgers} \nabla^{\phi}\phi = \phi_{y} + \phi \phi_{t}, \end{equation}
see \cite[(4.4)]{CFO}. Since $\phi,\phi_{t}$ are bounded, $\phi$ is $\calC^1(\R^2)$ and $\phi_{y},\phi_{t}$ are Euclidean $\alpha$-H\"older, we infer that $\nabla^{\phi}\phi$ is Euclidean $\alpha$-H\"older continuous. Since also $\phi_{y}$ is bounded, we obtain
\begin{align*} |\nabla^{\phi}\phi\big(y + y_{0},t + t_{0} + \phi(y_{0},t_{0})y\big) - \nabla^{\phi}\phi(y_{0},t_{0})| & \lesssim \min\{1,|(y,t + \phi(y_{0},t_{0})y)|^{\alpha}\}\\
& \lesssim \min\{1,|(y,t)|^{\alpha}\} \lesssim \|(y,t)\|^{\alpha}, \end{align*}
which by \eqref{form25} verifies that $\phi \in C^{1,\alpha}(\W)$.
\end{remark}

\begin{lemma}\label{l:C1ImpliesLipschitz} Assume that $\phi \in C^{1}(\W)$ with $\nabla^{\phi}\phi \in L^{\infty}(\W)$. Then, $\phi$ is intrinsic Lipschitz. \end{lemma}

\begin{proof} For simplicity, we assume that $\W$ is the $(y,t)$-plane. Write $L := \|\nabla^{\phi}\phi\|_{L^{\infty}(\W)}$. By Proposition 4.56(iii) in \cite{MR3587666}, it suffices to verify that
\begin{displaymath} |\phi^{(p^{-1})}(y,t)| = |\phi^{(p^{-1})}(y,t) - \phi^{(p^{-1})}(0,0)| \lesssim_{L} \|(y,t)\|, \qquad (y,t) \in \W, \: p \in \Gamma(\phi). \end{displaymath}
Write $p = w \cdot \phi(w)$. Then, we estimate as follows:
\begin{align*} |\phi^{(p^{-1})}(y,t)| & \leq |\phi^{(p^{-1})}(y,t) - \nabla^{\phi}\phi(w)y| + |\nabla^{\phi}\phi(w)y|\\
& \lesssim_{L} \|(y,t)\| + L\|(y,t)\|, \end{align*}
as claimed. The estimate leading to the last line follows from Proposition \ref{p:affineApproximation} (with $\alpha = 0$) below. \end{proof}

\begin{proposition}\label{p:affineApproximation} Fix $\alpha \in [0,1]$, and assume that $\W$ is the $(y,t)$-plane. Assume that $\phi \in C^{1,\alpha}(\W)$ with $L := \|\nabla^{\phi}\phi\|_{L^{\infty}(\W)} < \infty$. Then, for $p = w \cdot \phi(w) \in \Gamma(\phi)$,
\begin{displaymath} |\phi^{(p^{-1})}(y,t) - \nabla^{\phi}\phi(w)y| \lesssim \max \{ \|(y,t)\|^{1+\alpha} , \|(y,t)\|^{1 + \frac{\alpha}{2} } \},
 \qquad (y,t) \in \W, \end{displaymath}
where the implicit constants only depend on $L$, and, if $\alpha >
0$, also on the H\"older continuity constant "$H$" in the
definition of $C^{1,\alpha}(\W)$.
\end{proposition}
We postpone the proof to Section \ref{s:regularity}.

\begin{remark} If $\alpha > 0$, Proposition \ref{p:affineApproximation} above shows that functions in $C^{1,\alpha}(\W)$ with $\nabla^{\phi}\phi \in L^{\infty}(\W)$ are \emph{uniformly intrinsically differentiable}, see \cite[Definition 3.16]{zbMATH05598806}. \end{remark}

The next lemma verifies that being an intrinsic $C^{1,\alpha}$-graph is a left-invariant concept.
\begin{lemma}\label{l:translation} Let $\alpha > 0$. Let $\phi \in C^{1,\alpha}(\W)$ with constant "$H$" as in Definition \ref{d:C1,alpha,W}, and write $\Gamma := \Gamma(\phi)$. Fix $q \in \He$, and consider the function $\phi^{(q)}$, which parametrises the graph $q \cdot \Gamma$. Then, $\phi^{(q)} \in C^{1,\alpha}(\W)$ with the same constant "$H$".
\end{lemma}

\begin{proof} Let $\psi$ be the map $\phi^{(q)}$ which parametrises the translated graph $q \cdot \Gamma$. Since $\phi$ is by assumption everywhere intrinsically differentiable, and since intrinsic differentiability is a left-invariant notion, $\psi$ is intrinsically differentiable everywhere. It remains to verify for every $p_0 \in \Gamma(\psi)$ and for all $w\in \W$ that
\begin{equation}\label{eq:goal_psi} |\nabla^{\psi^{(p_{0}^{-1})}}\psi^{(p_{0}^{-1})}(w) - \nabla^{\psi^{(p_{0}^{-1})}}\psi^{(p_{0}^{-1})}(\mathbf{0})| \leq H\|w\|^{\alpha}. \end{equation}
By definition
$
\psi^{(p_0^{-1})} = (\phi^{(q)})^{(p_0^{-1})}
$
parametrises the graph $p_0^{-1} \cdot q \cdot \Gamma$ and hence $$\psi^{(p_0^{-1})}= \phi^{(p_0^{-1}\cdot q)}= \phi^{([q^{-1}\cdot p_0]^{-1})}.$$ Thus, denoting $p:= q^{-1}\cdot p_0$, the expression we wish to estimate, reads as follows:
\begin{displaymath}
 |\nabla^{\psi^{(p_{0}^{-1})}}\psi^{(p_{0}^{-1})}(w) - \nabla^{\psi^{(p_{0}^{-1})}}\psi^{(p_{0}^{-1})}(\mathbf{0})| =
 |\nabla^{\phi^{(p^{-1})}}\phi^{(p^{-1})}(w) - \nabla^{\phi^{(p^{-1})}}\phi^{(p^{-1})}(\mathbf{0})|.
\end{displaymath}
Since $p =  q^{-1}\cdot p_0$ is a point in $q^{-1} \cdot \Gamma(\psi) = q^{-1} \cdot q \cdot \Gamma = \Gamma$, the estimate \eqref{eq:goal_psi} then follows from the assumption $\phi \in C^{1,\alpha}(\mathbb{W})$, more precisely \eqref{holder} applied with $p$ instead of $p_0$.
\end{proof}

\subsubsection{Area formula}\label{s:areaFormula} We will need an area formula for functions $\phi \in C^{1}(\W)$. Since the intrinsic graphs of such functions are $H$-regular submanifolds by Remark \ref{r:Hregularity}, such a formula is available, due to Ambrosio, Serra Cassano and Vittone \cite{zbMATH05015569}. Specialised to our situation, the formula reads as follows: there exists a homogeneous left-invariant metric $d_{1}$ on $\He$ such that
\begin{equation}\label{form24} \calS^{3}_{d_{1}}(\Phi(\Omega)) = \int_{\Omega} \sqrt{1 + (\nabla^{\phi}\phi)^{2}} \, d\calL^{2} \end{equation}
for all $\phi \in C^{1}(\W)$ with graph map $\Phi$, and all open sets $\Omega \subset \W$. Here $\calS^{3}_{d_{1}}$ is the $3$-dimensional spherical measure defined via the distance $d_{1}$.
This is essentially the area formula we were after, but we will still record the following generalisation:

\begin{proposition}\label{p:area_formula}
There exists a left-invariant homogeneous distance $d_{1}$ on $\mathbb{H}$ such that the associated spherical Hausdorff measure $\mathcal{S}^3 = \calS^{3}_{d_{1}}$ satisfies
\begin{equation}\label{eq:area_form}
\int_{\Gamma} h \;d\mathcal{S}^3= \int_{\mathbb{W}} (h \circ \Phi) \sqrt{1+(\nabla^{\phi}\phi)^2}\, d\calL^{2}
\end{equation}
for every $\phi \in C^{1}(\W)$ with graph $\Gamma$, and for every $h\in L^1(\mathcal{S}^3|_{\Gamma})$. \end{proposition}

\begin{proof} This is standard: for any bounded open set $\Omega \subset \W$, formula \eqref{form24} implies the claim for the function $h = \chi_{\Phi(\Omega)}$. The formula for arbitrary $L^1$-functions $h$ follows by approximation. \end{proof}

\section{Proof of the main theorem}\label{s:mainProof} In this section, we prove the main result, Theorem \ref{mainIntro}, or rather its variant, Theorem \ref{mainABC}; recall that this is sufficient by the discussion preceding the statement of Theorem \ref{mainABC}. We fix the following data:
\begin{itemize}
\item a $3$-dimensional CZ kernel $K$ satisfying the AB condition,
\item a vertical subgroup $\W$ with complementary horizontal subgroup $\V$ (we will assume without loss of generality that $\W$ is the $(y,t)$-plane and $\V$ is the $x$-axis),
\item a function $\phi \in C^{1,\alpha}(\W)$, $\alpha > 0$, with compact support (recall Definition \ref{d:C1,alpha,W}), and
\item a $3$-ADR measure $\mu$ on $\Gamma := \Gamma(\phi)$.
\end{itemize}
The task is to show that the singular integral operator $T$ associated to $K$ is bounded on $L^{2}(\mu)$:
\begin{equation}\label{form3} \|T_{\mu,\epsilon}f\|_{L^{2}(\mu)} \leq A\|f\|_{L^{2}(\mu)}, \qquad f \in L^{2}(\mu), \: \epsilon > 0, \end{equation}
where $A \geq 1$ is a constant depending on the data above, but not on $\epsilon$. The first reduction is the following easy lemma:
\begin{lemma} Assume that \eqref{form3} holds for some $3$-ADR measure $\mu$ on $\Gamma$. Then \eqref{form3} holds (with possibly a different constant) for any $3$-ADR measure $\tilde{\mu}$ with $\spt \tilde{\mu} \subset \spt \mu$.
\end{lemma}

\begin{proof} Since $\mu$ is $3$-ADR, we clearly have $\mu = \varphi_{\mu} \,d \calH^{3}|_{\Gamma}$ for some $\varphi_{\mu} \in L^{\infty}(\calH^{3}|_{\Gamma})$.
Recall that $\calH^{3}|_{\Gamma}$ is $3$-ADR by \cite[Theorem 3.9]{MR3511465} since $\Gamma$ is the intrinsic graph of an intrinsic Lipschitz function according to Remark \ref{r:compactSupport}.
Therefore, $(\Gamma,\calH^{3})$ is a doubling metric measure space, where Lebesgue's differentiation theorem holds.
Thus
we also have that
\begin{displaymath} 1 \lesssim \limsup_{r \to 0} \frac{\mu(B(p,r))}{\calH^{3}(\Gamma \cap B(p,r))} = \limsup_{r \to 0} \frac{1}{\calH^{3}(\Gamma \cap B(p,r))} \int_{B(p,r) \cap \Gamma} \varphi_{\mu} \, d\calH^{3} = \varphi_{\mu}(p) \end{displaymath}
for $\calH^{3}$ almost every $p \in \spt \mu$, and in particular for $\mu$ almost every $p \in \spt \mu$. So, in the metric measure space $(\Gamma,\calH^{3})$, we have $\varphi_{\mu} \sim \chi_{\spt \mu}$. The same holds for $\tilde{\mu}$, by the same argument. Since $\spt \tilde{\mu} \subset \spt \mu$, it follows that we may write $\tilde{\mu} = g \, d\mu$ for some $g \in L^{\infty}(\mu)$ with $g \sim \chi_{\spt \tilde{\mu}}$. With this notation, we have
\begin{align*} T_{\tilde{\mu},\epsilon}f(p) & = \int_{\|q^{-1} \cdot p\| > \epsilon} f(q) K(q^{-1} \cdot p) \, d\tilde{\mu} (q)\\
& = \int_{\|q^{-1} \cdot p\| > \epsilon} f(q)g(q) K(q^{-1} \cdot p) \, d\mu (q) = T_{\mu,\epsilon}(fg)(p), \end{align*}
Finally,
\begin{displaymath} \|T_{\tilde{\mu},\epsilon}f(p)\|_{L^{2}(\tilde{\mu})} \lesssim \|T_{\mu,\epsilon}(fg)\|_{L^{2}(\mu)} \leq A\|fg\|_{L^{2}(\mu)} = A\|f\|_{L^{2}(\tilde{\mu})},  \end{displaymath}
as claimed. \end{proof}
The point of the lemma is that if we manage to prove \eqref{form3} for any single $3$-ADR measure $\mu$ with $\spt \mu = \Gamma$, then the same will follow for all $3$-ADR measures supported on $\Gamma$. In particular, it suffices to prove \eqref{form3} for the measure
\begin{equation}\label{muDef} \mu := \calS^{3}_{d_{1}}|_{\Gamma}, \end{equation}
which satisfies the area formula, Proposition \ref{p:area_formula}.

For this measure $\mu$, we prove \eqref{form3} by verifying the conditions of a suitable $T1$ theorem. To state these conditions, we use a system of dyadic cubes on $\Gamma$.
\subsection{Christ cubes and the $T1$ theorem}\label{s:T1} The following construction is due to Christ \cite{MR1096400}. For $j \in \Z$, there exists a family $\Delta_{j}$ of disjoint subsets of $\Gamma$ with the following properties:
\begin{itemize}
\item[(C0)] $\Gamma \subset \bigcup_{Q \in \Delta_{j}} \overline{Q}$,
\item[(C1)] If $j \leq k$, $Q \in \Delta_{j}$ and $Q' \in \Delta_{k}$, then either $Q \cap Q' = \emptyset$, or $Q \subset Q'$.
\item[(C2)] If $Q \in \Delta_{j}$, then $\diam Q \leq 2^{j} =: \ell(Q)$.
\item[(C3)] Every cube $Q \in \Delta_{j}$ contains a ball $B(z_{Q},c2^{j}) \cap \Gamma$ for some $z_{Q} \in Q$, and some constant $c > 0$.
\item[(C4)] Every cube $Q \in \Delta_{j}$ has thin boundary: there is a constant $D \geq 1$ such that $\mu(\partial_{\rho} Q) \leq D\rho^{\tfrac{1}{D}}\mu(Q)$, where
\begin{displaymath} \partial_{\rho} Q := \{q \in Q : \dist(q,\Gamma \setminus Q) \leq \rho \cdot \ell(Q)\}, \qquad \rho > 0. \end{displaymath}
\end{itemize}
The sets in $\Delta := \cup \Delta_{j}$ are called \emph{Christ cubes} (sometimes also \emph{David cubes} in the literature), or
just dyadic cubes, of $\Gamma$. It follows from (C2), (C3), and the $3$-regularity of $\mu$ that $\mu(Q) \sim \ell(Q)^{3}$ for $Q \in \Delta_{j}$.

To prove the $L^{2}$ boundedness of a CZ operator $T$ on $L^{2}(\mu)$, it suffices to verify the following conditions for a fixed system $\Delta$ of dyadic cubes on $\Gamma$:
\begin{equation}\label{T1} \|T_{\mu,\epsilon}\chi_{R}\|_{L^{2}(\mu|_{R})}^{2} \leq A\mu(R) \quad \text{and} \quad \|T_{\mu,\epsilon}^{\ast}\chi_{R}\|_{L^{2}(\mu|_{R})}^{2} \leq A\mu(R) \end{equation}
for all $R \in \Delta$, where $T_{\mu,\epsilon}^{\ast}$ is the formal adjoint of $T_{\mu,\epsilon}$, and $A \geq 1$ is a constant independent of $\epsilon$ and $R$. These conditions suffice for \eqref{form3} by the $T1$ theorem of David and Journ\'e, applied in the homogeneous metric measure space $(\Gamma,d,\mu)$, see \cite[Theorem 3.21]{tolsabook}. The statement in Tolsa's book is only formulated in Euclidean spaces, but the proof works the same way in homogeneous metric measure spaces; the details can be found in the honors thesis of Surath Fernando \cite{Fer}.

\subsection{Verifying the $T1$ testing condition}
\label{sec:vert1}
In this section, we use the $T1$ theorem to prove Theorem \ref{mainABC} (hence Theorem \ref{mainIntro}). The notation $\alpha,K,\phi,\Gamma,\W$ refers to the data fixed at the head of Section \ref{s:mainProof}, and $\mu$ is the measure in \eqref{muDef}. We start by remarking that it is sufficient to verify the first testing  condition, namely
\begin{equation}\label{T1_a}
 \|T_{\mu,\epsilon}\chi_{R}\|_{L^{2}(\mu|_{R})}^{2} \leq A\mu(R),\quad\text{for all } R\in \Delta,
 \end{equation}
This follows from the simple observation that $T^{\ast}$ has the same form as $T$, so our proof below for $T$ would equally well work for $T^{\ast}$. Let us be a bit more precise. Remark \ref{r:adjoint} shows that the kernel $K^{\ast}$ associated with the formal adjoint $T^{\ast}$ is also a CZ kernel. Moreover, since (i) the functions $\psi$ appearing in Definition \ref{abc} are radial, (ii) $\|w\|=\|w^{-1}\|$, and (iii) the measure $\mathcal{L}^2$ is invariant under the transformation $w\mapsto w^{-1}$ on $\W$, it follows that $K^{\ast}$ satisfies the AB condition with the same constant as $K$.

The first step in the proof of \eqref{T1_a} is a Littlewood-Paley decomposition of the operator $T$; the details in the Heisenberg group appeared in \cite{CL}, and we copy them nearly verbatim. Fix a smooth even function $\psi \colon \R \to \R$ with $\chi_{B(0,1/2)} \leq \psi \leq \chi_{B(0,2)}$, and then define the ($\He$-)radial functions $\psi_{j} \colon \He \to \R$ by
\begin{displaymath} \psi_{j}(p) := \psi(2^{j}\|p\|), \qquad j \in \Z. \end{displaymath}
Next, write $\eta_{j} := \psi_{j} - \psi_{j + 1}$ and $K_{(j)} := \eta_{j}K$, so that
\begin{equation}\label{supportKj} \spt K_{(j)} \subset B(\mathbf{0},2^{1 - j}) \setminus B(\mathbf{0},2^{-2 - j}). \end{equation}
We now consider the operators
\begin{displaymath} T_{(j)}f(p) = \int K_{(j)}(q^{-1} \cdot p)f(q) \, d\mu(q) \quad \text{and} \quad  S_{N} := \sum_{j \leq N} T_{(j)}, \end{displaymath}
\begin{remark}\label{r:K_j_Hol}
It is easy to check that the kernel $K_{(j)}$ satisfies the same growth and H\"older continuity estimates, namely \eqref{form0}, as $K$. In particular, by \eqref{supportKj} $K_{(j)} \in L^{\infty}(\He)$ with $\|K_{(j)}\|_{L^{\infty}(\He)} \lesssim 2^{3j}$.
\end{remark}
The next lemma demonstrates that $T_{\mu,\epsilon}$ and $S_{N}$ are very close to each other, for $\epsilon \in [2^{-N},2^{-N + 1})$:
\begin{lemma}
\label{l:smoothcomp} Fix $N \in \Z$ and $\epsilon \in [2^{-N},2^{-N + 1})$. Then
\begin{displaymath} |S_{N}f(p) - T_{\mu,\epsilon}f(p)| \lesssim M_{\mu}f(p), \qquad f \in L_{loc}^{1}(\mu), \end{displaymath}
where $M_{\mu}$ is the centred Hardy-Littlewood maximal function associated with $\mu$. \end{lemma}
\begin{proof} We first observe that
\begin{align*}\sum_{j \leq N} K_{(j)}(p) & = K(p) \sum_{j \leq N} \eta_{j}(p) = K(p) \cdot (1 - \psi_{N + 1}(p)), \qquad p \in \He \setminus \{\mathbf{0}\}. \end{align*}
Hence, for $\epsilon \in [2^{-N},2^{-N + 1})$,
\begin{align*} |S_{N}f(p) - T_{\mu,\epsilon}f(p)| & = \left|\int K(q^{-1} \cdot p)[(1 - \psi_{N + 1}) - \chi_{\He \setminus B(0,\epsilon)}](q^{-1} \cdot p)f(q) \, d\mu (q) \right|\\
& \leq \int_{B(p,2^{-N + 1}) \setminus B(p,2^{-N - 2})} |K(q^{-1} \cdot p)||f(q)| \, d\mu (q)\\
& \lesssim \frac{1}{\mu(B(p,2^{-N + 1}))} \int_{B(p,2^{-N + 1})} |f(q)| \, d\mu (q) \lesssim M_{\mu}f(p). \end{align*}
using the growth condition \eqref{form0}. \end{proof}
By the lemma above, and the $L^{2}$-boundedness of $M_{\mu}$, we have
\begin{displaymath} \|T_{\mu,\epsilon}\chi_{R}\|_{L^{2}(\mu|_{R})} \lesssim \|M_{\mu}\chi_{R}\|_{L^{2}(\mu)} + \|S_{N}\chi_{R}\|_{L^{2}(\mu|_{R})} \lesssim \mu(R)^{1/2} + \|S_{N}\chi_{R}\|_{L^{2}(\mu|_{R})}. \end{displaymath}
So, it remains to prove the following proposition:
\begin{proposition}\label{p:T1} For any $R \in \Delta$ and $N \in \Z$, we have $\|S_{N}\chi_{R}\|_{L^{2}(\mu|_{R})}^{2} \lesssim \mu(R)$. \end{proposition}

\begin{proof} We fix $R \in \Delta$ and $N \in \N$ for the rest of the proof. We start with the estimate
\begin{equation}\label{eq:start} \|S_N\chi_{R}\|_{L^{2}(\mu|_{R})} \leq \Big\| \sum_{2^{-N} \leq 2^{-j} \leq 4\ell(R)} T_{(j)}\chi_{R} \Big\|_{L^{2}(\mu|_{R})} + \sum_{2^{-j} > 4\ell(R)} \|T_{(j)}\chi_{R}\|_{L^{2}(\mu|_{R})} \end{equation}
We quickly deal with the terms in the second sum. Recall from \eqref{supportKj} that the support of
\begin{displaymath} q \mapsto K_{(j)}(q^{-1} \cdot p), \qquad p \in R, \end{displaymath}
is contained in $\He \setminus B(p,2^{-2 - j}) \subset \He \setminus R$, assuming $2^{-j} > 4\ell(R)$. Hence $T_{(j)}\chi_{R}(p) = 0$ for all $p \in R$ and $2^{-j} > 4\ell(R)$. So,
\begin{equation}\label{eq:start2}\sum_{2^{-j} > 4\ell(R)} \|T_{(j)}\chi_{R}\|_{L^{2}(\mu|_{R})} = 0. \end{equation}
Thus, it remains to study the term
\begin{displaymath} \|S \chi_{R} \|_{L^{2}(\mu|_{R})}, \quad \text{where} \quad S := \sum_{2^{-N} \leq 2^{-j} \leq 4\ell(R)} T_{(j)}. \end{displaymath}
We stress that the operator $S$ depends both on $N$ and $R$, but to avoid heavy notation, we refrain from explicitly marking this dependence. All the implicit multiplicative constants which appear in the following estimates will be uniform in $N$ and $R$.

The strategy for bounding $\|S \chi_{R}\|_{L^{2}(\mu|_{R})}$ is straightforward: we fix a point $p \in R$, and attempt to find an estimate for $S \chi_{R}(p)$. The quality of this estimate will depend on the choice of $p$ as follows: the closer  $p$ is to the ``boundary'' of $R$, in the sense that $\dist(p,\Gamma \setminus R)$ is small, the worse the estimate.
Motivated by this discussion, we define
\begin{displaymath} \partial_{\rho} R = \{q \in R : (\rho/2) \cdot \ell(R) < \dist(q,\Gamma \setminus R) \leq \rho \cdot \ell(R)\}, \qquad \rho > 0, \end{displaymath}
and recall that $\mu(\partial_{\rho} R) \lesssim \rho^{\tfrac{1}{D}}\mu(R)$ by (C4) from Section \ref{s:T1} (the notation we use here is a little different from Section \ref{s:T1}, in that we impose a two-sided inequality in the definition of $\partial_{\rho} R$). Also note that $\partial_{\rho} R = \emptyset$ for $\rho > 2$. Write $\rho(k) := 2^{1 - k}/\ell(R)$, and decompose $\|S \chi_{R}\|_{L^{2}(\mu|_{R})}^{2}$ as follows:
\begin{displaymath} \|S \chi_{R}\|_{L^{2}(\mu|_{R})}^{2} = \sum_{2^{-k} \leq \ell(R)} \int_{\partial_{\rho(k)} R} |S \chi_{R}(p)|^{2} \, d\mu(p) =: \sum_{2^{-k} \leq \ell(R)} I_{k}. \end{displaymath}
The task is now to estimate the terms $I_{k}$ separately. Note that whenever $p \in \partial_{\rho(k)} R$, then
\begin{equation}\label{form17} T_{(j)}\chi_{R}(p) = \int_{R} K_{(j)}(q^{-1} \cdot p) \, d\mu(q)  = \int_{\Gamma} K_{(j)}(q^{-1} \cdot p) \, d\mu(q), \quad j > k, \end{equation}
since the support of $q \mapsto K_{(j)}(q^{-1} \cdot p)$ lies in $B(p,2^{1 - j}) \subset B(p,2^{- k})$ and moreover, $\mathrm{dist}(p,\Gamma \setminus R)\geq 2^{-k}$ for $p \in  \partial_{\rho(k)} R$, therefore $ B(p,2^{- k}) \subset R$ for such $p$. We would now like to split the operator $S$  into two parts -- depending on $k$: a sum of terms where the estimate \eqref{form17} is valid (corresponding to indices $j$ with $2^{-j}\leq 2^{-k}$), and a sum with all the remaining terms. However, if $2^{-k} > 1$, we wish to further split the first sum into two parts, which we will discuss separately. Thus let us momentarily fix $k\in \Z$ with $2^{-k}\leq \ell(R)$ and consider the following $k$-dependent splitting:
\begin{align*} S = S_{I} + S_{II} + S_{III}= \sum_{j \in E_1(k)} T_{(j)} +
\sum_{j \in E_2(k)} T_{(j)} +
\sum_{j \in E_3(k)} T_{(j)}, \end{align*}
where
\begin{align*}
E_1(k)&:= \{j\in \Z:\,2^{-N} \leq 2^{-j} < \min\{1,2^{-k}\}\}\\
E_2(k)&:= \{j\in \Z:\,\min\{1,2^{-k}\} \leq 2^{-j} < 2^{-k}\}\\
E_3(k)&:= \{j\in \Z:\, 2^{-k} \leq 2^{-j} \leq 4\ell(R)\}.
\end{align*}
If $2^{-k}\leq 1$, then $E_{2}(k) = \emptyset$ and we simply have $S=S_I+ S_{III}$. Hence,
\begin{align}\label{play7} \sum_{2^{-k} \leq \ell(R)} I_{k} & \lesssim \sum_{2^{-k} \leq \ell(R)} \int_{\partial_{\rho(k)} R} |S_{I}\chi_{R}(p)|^{2} \, d\mu(p)\\
&\label{play5} + \sum_{1 < 2^{-k} \leq \ell(R)} \int_{\partial_{\rho(k)} R} |S_{II}\chi_{R}(p)|^{2} \, d\mu(p)\\
&\label{play6} + \sum_{2^{-k} \leq \ell(R)} \int_{\partial_{\rho(k)} R} |S_{III}\chi_{R}(p)|^{2} \, d\mu(p) \end{align}
On all three lines, we need to get $\lesssim \mu(R)$ on the right hand side; we start with line \eqref{play6}. The pointwise estimate we can obtain for $S_{III}\chi_{R}(p)$, for $p \in \partial_{\rho(k)} R$, is fairly lousy: it is based on the trivial estimate
\begin{equation}\label{form16} |T_{(j)}f(p)| \lesssim \|f\|_{L^{\infty}(\mu)}, \qquad j \in \Z, \: p \in \He, \:  f \in L^{\infty}(\mu). \end{equation}
This follows by observing that the support of
\begin{equation}\label{support} q \mapsto K_{(j)}(q^{-1} \cdot p) \end{equation}
is contained in the annulus $B(p,2^{1 - j}) \setminus B(p,2^{-2 - j})$, so $|K_{(j)}(q^{-1} \cdot p)| \lesssim 2^{3j}$, and finally
\begin{displaymath} |T_{(j)}f(p)| \lesssim \|f\|_{L^{\infty}(\mu)} \int_{B(p,2^{1 - j})} 2^{3j} \, d\mu \lesssim \|f\|_{L^{\infty}(\mu)}. \end{displaymath}
In particular, \eqref{form16} implies that
\begin{displaymath}  |S_{III}\chi_{R}(p)| \lesssim C + \log \frac{\ell(R)}{2^{-k}} \sim C + \log \frac{1}{\rho(k)}, \qquad p \in \partial_{\rho(k)} R. \end{displaymath}
Recalling again that $\mu(\partial_{\rho} R) \lesssim \rho^{\tfrac{1}{D}} \mu(R)$, we obtain
\begin{equation*}
\int_{\partial_{\rho(k)} R} |S_{III}\chi_{R}(p)|^{2} \, d\mu(p)\lesssim
\rho(k)^{\tfrac{1}{D}} \left(C + \log \frac{1}{\rho(k)} \right)^{2} \mu(R). \end{equation*}
The last terms are summable to $\lesssim \mu(R)$ over the range $2^{-k} \leq \ell(R)$ (which is equivalent to $\rho(k) \leq 2$), so we are done with estimating line \eqref{play6}.

In estimating $S_{I}$ and $S_{II}$, the following observation is crucial. Since \eqref{form17} holds for all $2^{-j} < 2^{-k}$, its analogue also holds for $S_{I}$ and $S_{II}$:
\begin{displaymath} S_{I}\chi_{R}(p) = \int_{\Gamma} K_{I}(q^{-1} \cdot p) \, d\mu(q) \quad \text{and} \quad S_{II}\chi_{R}(p) = \int_{\Gamma} K_{II}(q^{-1} \cdot p) \, d\mu(q) \end{displaymath}
for $p \in \partial_{\rho(k)} R$, where
\begin{displaymath} K_{I} := \sum_{j \in E_{1}(k)} K_{(j)} \quad \text{and} \quad K_{II} := \sum_{j \in E_{2}(k)} K_{(j)}. \end{displaymath}

We will now deal with line \eqref{play5}. Fix $k$ with $2^{-k} > 1$. We claim that if $p \in \partial_{\rho(k)} R \setminus B(0,C)$, where $C = C(\phi) \geq 1$ is a suitable constant, then $|S_{II}\chi_{R}(p)|$ is bounded by another constant, which only depends on the annular boundedness condition. This will give the estimate
\begin{equation}\label{play1} \int_{\partial_{\rho(k)} R} |S_{II}\chi_{R}(p)|^{2} \, d\mu(p) \lesssim \int_{B(0,C)} |S_{II}\chi_{R}(p)|^{2} \, d\mu(p) + \mu(\partial_{\rho(k)} R), \end{equation}
which will be good enough.

Now, assume that $\partial_{\rho(k)} R \setminus B(0,C) \neq \emptyset$, and fix $p = w_{0} \cdot \phi(w_{0}) \in \partial_{\rho(k)} R \setminus B(0,C)$. If $C$ was chosen large enough, the compact support of $\phi$ implies that
\begin{displaymath} p = w_{0} \in \W. \end{displaymath}
By \eqref{form17} and the area formula, Proposition \ref{p:area_formula},
\begin{displaymath} S_{II}\chi_{R}(p) = \int_{\Gamma} K_{II}(q^{-1} \cdot p) \, d\mu(q) = \int_{\W} K_{II}(\Phi_{\Gamma}(y,t)^{-1} \cdot w_{0})\sqrt{1 + \nabla^{\phi}\phi(y,t)^{2}} \, dy \, dt. \end{displaymath}
Write $\Phi_{\W}(w) = w$ for the graph map parametrizing $\W$ itself. Then, by the annular boundedness assumption, and the area formula again, it follows that
\begin{equation}
\begin{split}
\label{telesc} \left|\int_{\W} K_{II}(\Phi_{\W}(y,t)^{-1} \cdot w_{0}) \, dy \, dt \right| & = \left|\int_{\W} K_{II}(w^{-1} \cdot w_{0}) \, d\mathcal{S}^3(w)\right|\\
& = \left|\int_{\W} K_{II}(w) \, d\mathcal{S}^3(w)\right| \leq A. \end{split}\end{equation}
To justify the last inequality, we write
\begin{equation*}
\begin{split}
 \left|\int_{\W} K_{II}(w) \, d\mathcal{S}^3(w)\right|=&\left| \int_{\W} \sum_{j \in E_2(k)} K_{(j)}(w)\, d \mathcal{S}^3(w) \right| \\
 &=\left| \int_{\W} \sum_{k<j \leq 0} \eta_j(w) K(w)\, d \mathcal{S}^3(w) \right| \\
 & =\left| \int_{\W}(\psi(2^{-(k+1)} \|w \|)-\psi(\|w\|) )K(w) \, d \mathcal{S}^3(w)\right|.
\end{split}
\end{equation*}
Observe that the function $\psi':\He \ra \R$, defined by $\psi'(w)=\psi (\|w\|)$, satisfies the conditions from Definition \ref{abc} and  $\psi'^r(w)=\psi'(\delta_{r^{-1}}(w))=\psi(\|w\|/r)$ for $w \in \W$ and $r>0$. The careful reader may have noticed that Definition \ref{abc} has been formulated in terms of an integral with respect to the $2$-dimensional Lebesgue measure $\mathcal{L}^2$, rather than $\mathcal{S}^3$. However, restricted to $\W$, Heisenberg group multiplication behaves like addition in $\mathbb{R}^2$, see \eqref{eq:group_law}, and so $\mathcal{L}^2$ yields a uniformly distributed measure on $(\W,d)$. Since also
$\mathcal{S}^3$ restricted to $\W$ is a uniformly distributed measure, the two measures $\mathcal{L}^2$ and $\mathcal{S}^3$ agree up to a multiplicative constant,
see Theorem 3.4 in \cite{zbMATH01249699}. Moreover, this constant does not depend on the choice of $\W$ since rotations around the vertical axis are isometries both for the Euclidean and the Kor\'{a}nyi distance.
In conclusion,  \eqref{telesc} follows by \eqref{vabann}.

Consequently, $S_{II}\chi_{R}(p)$ only differs in absolute value by $\leq A$ from the following expression:
\begin{equation}\label{play2} \int_{\W} \left[K_{II}(\Phi_{\Gamma}(y,t)^{-1} \cdot w_{0})\sqrt{1 + \nabla^{\phi}\phi(y,t)^{2}} - K_{II}(\Phi_{\W}(y,t)^{-1} \cdot w_{0}) \right] \, dy \, dt. \end{equation}
We immediately note that the integrand vanishes identically for $(y,t) \in \W \setminus \spt \phi$, since $\Phi_{\Gamma}(y,t) = \Phi_{\W}(y,t)$ and $\nabla^{\phi}\phi(y,t) = 0$ for such $(y,t)$. What if $(y,t) \in \spt \phi$? Note that if the integrand does not vanish, then, by the definition of $K_{II}$, we have
\begin{displaymath} \Phi_{\Gamma}(y,t)^{-1} \cdot w_{0} \in \bigcup_{1 < 2^{-j} \leq 2^{-k}} \spt K_{(j)} \quad \text{or} \quad \Phi_{\W}(y,t)^{-1} \cdot w_{0} \in \bigcup_{1 < 2^{-j} \leq 2^{-k}} \spt K_{(j)}. \end{displaymath}
Recall from \eqref{supportKj} that $\spt K_{(j)} \subset B(\mathbf{0},2^{1 - j}) \setminus B(\mathbf{0},2^{-2 - j})$. Hence, if $\Phi_{\Gamma}(y,t)^{-1} \cdot w_{0} \in \spt K_{(j)}$ for instance, we have
\begin{equation}\label{form36} p = w_{0} \in B(\Phi_{\Gamma}(y,t),2^{1 - j}) \setminus B(\Phi_{\Gamma}(y,t),2^{-2 - j}) \subset B(\Phi_{\Gamma}(y,t),2^{1 - j}). \end{equation}
Given that $\Phi_{\Gamma}(y,t)$ always lies in the fixed ball $B(0,\diam(\Phi_{\Gamma}(\spt \phi)))$, independent of the particular choice of $(y,t) \in \spt \phi$, and $p \in \He \setminus B(0,C)$, we infer that \eqref{form36} can only occur for $\lesssim 1$ indices $j$ with $2^{-j} > 1$ (the particular indices naturally depend on the location of our fixed point $p = w_{0}$). Finally, noticing that $\|K_{II}\|_{L^{\infty}} \lesssim 1$ (see Remark \ref{r:K_j_Hol} and note that $E_{2}(k)$ only contains negative indices $j$), we see that the expression in \eqref{play2} is bounded by $\lesssim \calL^{2}(\spt \phi) \lesssim 1$. This proves that $|S_{II}\chi_{R}(p)| \lesssim 1$, and establishes \eqref{play1}.

Finally, we observe that for $p \in B(0,C)$, we have the trivial estimate $|S_{II}\chi_{R}(p)| \lesssim 1 + \log \ell(R)$, using \eqref{form16} (note that since $2^{-k} > 1$, also $\ell(R) \geq 2^{-k} > 1$). Combined with \eqref{play1}, and $\mu(B(0,C)) \lesssim 1$, we get
\begin{displaymath} \int_{\partial_{\rho(k)} R} |S_{II}\chi_{R}(p)|^{2} \, d\mu(p) \lesssim (1 + \log \ell(R))^{2} + \mu(\partial_{\rho(k)} R). \end{displaymath}
Hence
\begin{displaymath} \eqref{play5} \lesssim \sum_{1 \leq 2^{-k} \leq \ell(R)} [(1 + \log \ell(R))^{2} + \mu(\partial_{\rho(k)} R)] \lesssim \ell(R)^{3} \sim \mu(R), \end{displaymath}
as desired.

It remains to estimate the ``main term'' on line \eqref{play7}. The plan is simply to give a point-wise estimate for the integrand $|S_{I}\chi_{R}(p)|^{2}$, for $p \in \partial_{\rho(k)} R$. Fix $p \in \partial_{\rho(k)} R$, and recall, by \eqref{form17} and the definition of $S_{I}$, that
\begin{equation}\label{eq:S_I_small_translate} S_{I}\chi_{R}(p) = \int_{\Gamma} K_{I}(q^{-1} \cdot p) \, d\calS^{3}(q) = \int_{(p^{-1}) \cdot \Gamma} K_{I}(q^{-1}) \, d\calS^{3}(q). \end{equation}
Since $(p^{-1}) \cdot \Gamma$ is a graph with the same properties as $\Gamma$, we may assume that $p = \mathbf{0}$. More precisely, in this last part of the proof, we will only rely on the $C^{1,\alpha}$-hypothesis (and not the compact support of $\phi$), which is left-translation invariant by Lemma \ref{l:translation}.

So, we assume without loss of generality that $p = \mathbf{0}$ (and thus $\mathbf{0}\in \Gamma$ and $\phi(\mathbf{0})=\mathbf{0}$).
Hence, applying the area formula, Proposition \ref{p:area_formula}, the task is to estimate
\begin{equation}\label{form21} \left| \int_{\Gamma} K_{I}(q^{-1}) \, d\mu(q) \right| = \left| \int_{\W} K_{I}(\Phi_{\Gamma}(y,t)^{-1}) \sqrt{1 + \nabla^{\phi}\phi(y,t)^{2}} \, dy \, dt \right|. \end{equation}
The plan is to approximate $\Gamma$ around $p = \mathbf{0}$ by a vertical plane $\W_{0}$, namely the vertical tangent plane of $\Gamma$ at $\mathbf{0}$, and use the annular boundedness property. So, let $\W_{0}$ be the intrinsic graph of the function $\phi_{0}(y,t) = \nabla^{\phi}\phi(0,0)y$, note that $\nabla^{\phi_{0}}\phi_{0} \equiv \nabla^{\phi}\phi(0,0)$, and let $\Phi_{\W_{0}}(y,t) = (y,t) \cdot \phi_{0}(y,t)$ be the graph map of $\phi_{0}$.
By the annular boundedness as in \eqref{telesc}
and the area formula again, we obtain
\begin{displaymath} \left|\int_{\W} K_{I}(\Phi_{\W_{0}}(y,t)^{-1})\sqrt{1 + \nabla^{\phi}\phi(0,0)^{2}} \, dy \, dt\right| = \left|\int_{\W_{0}} K_{I}(q^{-1}) \, d\mathcal{S}^3(q) \right| \leq A \end{displaymath}
Then, we estimate as follows:
\begin{align} \left| \int_{\Gamma} K_{I}(q^{-1}) \, d\mu(q) \right| & \leq \left| \int_{\Gamma} K_{I}(q^{-1})\, d \mu(q) - \int_{\W_0} K_{I}(q^{-1}) \, d\mathcal{S}^3(q) \right| + A\notag\\
&\leq \Big| \int_{\W} K_{I}(\Phi_{\Gamma}(y,t)^{-1})\sqrt{1 + \nabla^{\phi}\phi(y,t)^{2}} \, dy \, dt \notag\\
& - \int_{\W} K_{I}(\Phi_{\W_{0}}(y,t)^{-1}) \sqrt{1 + \nabla^{\phi}\phi(0,0)^{2}} \, dy \, dt \Big| + A\notag\\
\leq &\label{form33} \int_{\W}\left|K_{I}(\Phi_{\Gamma}(y,t)^{-1})-K_{I}(\Phi_{\W_{0}}(y,t)^{-1}) \right| \sqrt{1 + \nabla^{\phi}\phi(y,t)^{2}} \, dy \, dt \\
 &\label{form38} +  \Big| \int_{\W} K_{I}(\Phi_{\W_{0}}(y,t)^{-1})\left(\sqrt{1 + \nabla^{\phi}\phi(y,t)^{2}}  -\sqrt{1 + \nabla^{\phi}\phi(0,0)^{2}}\right) \, dy \, dt \Big|\\
 &+ A.\notag
\end{align}

Next, we recall that $K_{I}$ is a finite sum of the kernels $K_{(j)}$. By the triangle inequality,
\begin{align*} & \left|K_{I}(\Phi_{\Gamma}(y,t)^{-1})-K_{I}(\Phi_{\W_{0}}(y,t)^{-1}) \right|
 \leq \sum_{j\in E_1(k)} \left|K_{(j)}(\Phi_{\Gamma}(y,t)^{-1})-K_{(j)}(\Phi_{\W_{0}}(y,t)^{-1}) \right|\end{align*}
and
\begin{displaymath} |K_{I}(\Phi_{\W_{0}}(y,t)^{-1})| \leq \sum_{j\in E_1(k)} |K_{(j)}(\Phi_{\W_{0}}(y,t)^{-1})|. \end{displaymath}
We now split the terms in \eqref{form33} and \eqref{form38} to sums over the indices $j\in E_1(k)$, and estimate the terms
\begin{equation}\label{play3} \int_{\W}\left|K_{(j)}(\Phi_{\Gamma}(y,t)^{-1})-K_{(j)}(\Phi_{\W_{0}}(y,t)^{-1}) \right| \sqrt{1 + \nabla^{\phi}\phi(y,t)^{2}} \, dy \, dt \end{equation}
 and
\begin{equation}\label{play4} \int_{\W} |K_{(j)}(\Phi_{\W_{0}}(y,t)^{-1})| \left| \sqrt{1 + \nabla^{\phi}\phi(y,t)^{2}}  -\sqrt{1 + \nabla^{\phi}\phi(0,0)^{2}}\right| \, dy \, dt\end{equation}
 for fixed $j\in E_1(k)$. We recall from \eqref{supportKj} that $\spt K_{(j)} \subset B(\mathbf{0},2^{1 - j}) \setminus B(\mathbf{0},2^{-2 - j})$, so we can restrict in both \eqref{play3} and \eqref{play4} the domain of integration to
\begin{displaymath} \pi_{\mathbb{W}}(B(\mathbf{0},2^{1-j}) \setminus B(\mathbf{0},2^{-2 - j})). \end{displaymath}
To make use of this, we record the formula for the vertical projection $\pi_{\W}$:
\begin{displaymath} \pi_{\W}(x,y,t) = (0,y,t + \tfrac{1}{2}xy). \end{displaymath}
In particular, if $q = (x,y,t) \in B(\mathbf{0},2^{1 - j})$, then
\begin{equation}\label{form34} \|\pi_{\W}(q)\| \lesssim |y| + |t + \tfrac{1}{2}xy|^{1/2} \lesssim 2^{-j}. \end{equation}
In order to bound the term
in \eqref{play3}, we use the H\"older continuity of $K_{(j)}$, which is provided by Remark \ref{r:K_j_Hol} from the corresponding estimate for $K$. Instead of applying directly the growth condition from \eqref{form0}, it is slightly more convenient to resort to the estimate we obtain from Lemma \ref{l:standardHolder}.
Fix $(y,t) \in \pi_{\W}(B(\mathbf{0},2^{1 - j}) \setminus B(\mathbf{0},2^{-2 - j}))$, so that $\|(y,t)\| \lesssim 2^{-j}$ by \eqref{form34}.
Applying the estimate in Lemma \ref{l:standardHolder} with $w= \mathbf{0}$, $v_1= \Phi_{\Gamma}(y,t)$ and $v_2= \Phi_{\W_{0}}(y,t)$, we find  that
\begin{equation}\label{form45}
\left| K_{(j)}(\Phi_{\Gamma}(y,t)^{-1})-K_{(j)}(\Phi_{\W_{0}}(y,t)^{-1}) \right|
\lesssim \frac{\|\Phi_{\W_{0}}(y,t)^{-1}\cdot \Phi_{\Gamma}(y,t)\|^{\beta/2}}{\|\Phi_{\Gamma}(y,t)\|^{3+\beta/2}}
\end{equation}
This estimate is legitimate, if $d(v_{1},v_{2}) \leq \|v_{1}\|/2$. This may not be the case to begin with, but then we know that $2^{-j} \lesssim \|v_{1}\|/2 \leq d(v_{1},v_{2}) \lesssim 2^{-j}$, and we can pick boundedly many points $v_{1}',\ldots,v_{n}'$ with the following properties: $v_{1}' = v_{1}$, $v_{n}' = v_{2}$, $\|v_{i}'\| \sim 2^{-j}$ and $2^{-j} \sim d(v_{i}',v_{i + 1}') \leq \|v_{i}'\|/2$; in particular $d(v_{i}',v_{i + 1}') \lesssim d(v_{1},v_{2})$. Then, we  obtain \eqref{form45} by the triangle inequality, and boundedly many applications of the H\"older estimates.

We then continue \eqref{form45} by using the definition of graph
maps, and Proposition \ref{p:affineApproximation} in the case
$\|(y,t)\|\lesssim 2^{-j}\leq 1$ (since we are currently dealing
with $j\in E_1(k)$):
\begin{displaymath} \|\Phi_{\W_{0}}(y,t)^{-1}\cdot \Phi_{\Gamma}(y,t)\| = |\nabla^{\phi}\phi(0,0)y-\phi(y,t)| \lesssim_{H,L} \|(y,t)\|^{1 +
\tfrac{\alpha}{2}} \lesssim 2^{-j(1 + \tfrac{\alpha}{2})}.
\end{displaymath} Hence, by \eqref{form45}, and the estimate
above,
\begin{displaymath} \left| K_{(j)}(\Phi_{\Gamma}(y,t)^{-1})-K_{(j)}(\Phi_{\W_{0}}(y,t)^{-1})
\right| \lesssim_{H,L} 2^{3j - \alpha \beta j/4},
\end{displaymath} for $(y,t) \in \pi_{\W}(B(\mathbf{0},2^{1 - j})
\setminus B(\mathbf{0},2^{-2 - j}))$. Now, we can finish the bound
for \eqref{play3} by observing that
\begin{equation}\label{form39} \calL^{2}(\pi_{\W}(B(\mathbf{0},2^{1 - j}))) \lesssim 2^{-3j} \end{equation}
by \eqref{form34}, and hence
\begin{displaymath} \eqref{play3} \lesssim_{H,L} (1 + \|\nabla^{\phi}\phi\|_{L^{\infty}(\W)})
\calL^{2}(\pi_{\W}(B(\mathbf{0},2^{1 - j}))) \cdot 2^{3j - \alpha
\beta j/4} \lesssim_{H,L} 2^{-\alpha \beta j/4}. \end{displaymath}
Next, we desire a similar estimate for \eqref{play4}, but this is
easier, using the fact, see Remark \ref{r:K_j_Hol}, that
$\|K_{(j)}\|_{L^{\infty}(\He)} \lesssim 2^{3j}$. Hence, by the
definition of $\phi \in C^{1,\alpha}(\W)$, and the the estimates
\eqref{form34}, \eqref{form39},
\begin{align*} \eqref{play4} \lesssim & 2^{3j}\int_{\pi_{\W}(B(\mathbf{0},2^{1 - j}))}
 |\nabla^{\phi}\phi(y,t) - \nabla^{\phi}\phi(0,0)| \, dy \, dt\\
& \lesssim_{H} 2^{3j} \int_{\pi_{\W}(B(\mathbf{0},2^{1 - j}))} \|(y,t)\|^{\alpha} \, dy \, dt \lesssim 2^{-\alpha j}. \end{align*}

Next we insert the bounds for \eqref{play3}  and  \eqref{play4} back into  \eqref{form33} and \eqref{form38} (recalling also the reduction made at \eqref{eq:S_I_small_translate}):
\begin{displaymath}
\left|S_I \chi_R(p)\right| \lesssim A + \sum_{j\in E_1(k)}
\left(2^{-\alpha \beta j/4}+2^{-\alpha j}\right) \lesssim 1, \quad
p \in \partial_{\rho(k)} R.
\end{displaymath}
Here the sum is uniformly bounded since by definition $E_1(k)$ is the set of indices $j$ with $2^{-N} \leq 2^{-j} \leq \min\{1,2^{-k}\}$.
Finally,
\begin{align*}
\int_{\partial \rho_{(k)} R} \left|S_I \chi_R(p)\right|^2 \, d\mu(p) & \lesssim \mu\left(\partial_{\rho(k)} R\right)
 \lesssim \left(\rho(k)\right)^{\frac{1}{D}} \mu(R),
\end{align*}
which shows that $\eqref{play7} \lesssim \mu(R)$. This completes the proof of Proposition \ref{p:T1}. \end{proof}

By the $T1$ theorem, see \eqref{T1}, we have now established that $T$ is bounded on $L^{2}(\mu)$, as stated in Theorem \ref{mainIntro}.

\section{H\"older regularity and linear approximation}\label{s:regularity}

In this section, we prove Proposition \ref{p:affineApproximation}, which we restate below for convenience:

\begin{proposition}\label{p:affineApproximation2} Fix $\alpha \in [0,1]$, and assume that $\W$ is the $(y,t)$-plane. Assume that $\phi \in C^{1,\alpha}(\W)$ with $L := \|\nabla^{\phi}\phi\|_{L^{\infty}(\W)} < \infty$. Then, for $p = w \cdot \phi(w) \in \Gamma(\phi)$,
\begin{displaymath} |\phi^{(p^{-1})}(y,t) - \nabla^{\phi}\phi(w)y| \lesssim\max \{\|(y,t)\|^{1+\alpha}, \|(y,t)\|^{1 + \tfrac{\alpha}{2}}\}, \qquad (y,t) \in \W, \end{displaymath}
where the implicit constant depends on $L$, and, if $\alpha > 0$, also on the H\"older continuity constant ``$H$'' in the definition of $C^{1,\alpha}(\W)$.
\end{proposition}

The proof depends on a non-trivial H\"older estimate which functions $\phi$ as in Proposition \ref{p:affineApproximation2} satisfy along vertical lines in $\W$, see Proposition \ref{p:improvedRegularity}. Momentarily taking this estimate for granted, we explain now how to deduce Proposition \ref{p:affineApproximation2}.

\begin{proof} By definition of $\nabla^{\phi}\phi$, we have
\begin{displaymath} |\phi^{(p^{-1})}(y,t) - \nabla^{\phi}\phi(w)y| = |\phi^{(p^{-1})}(y,t) - \nabla^{\phi^{(p^{-1})}}\phi^{(p^{-1})}(0,0)y|. \end{displaymath}
So, we simply want to prove that $|\phi(y,t) - \nabla^{\phi}\phi(0,0)y| \lesssim \|(y,t)\|^{1 + \alpha}$ under the assumption that $p = \mathbf{0}$ and $\phi(0,0) = 0$ (the constants $L$ and $H$ are not changed under left translations).

Consider the following ODE:
\begin{equation}\label{ODEa} \begin{cases} \tau'(s) = \phi(s,\tau(s)), \\ \tau(0) = 0. \end{cases} \end{equation}
Since $\phi$ is intrinsically differentiable, $\phi$ is continuous, see Proposition 4.74 in \cite{MR3587666}. So, by Peano's theorem, the equation \eqref{ODEa} has a (possibly non-unique) solution, which exists on some maximal interval $J$ containing $0$. Moreover, $(s,\tau(s))$ leaves any compact set of the plane $\W$, as $s$ tends to either endpoint of $J$; see for instance
\cite[Corollary 2.16]{MR2961944} or \cite[Theorem 2.1]{MR587488}.
We presently want to argue that $J = \R$. Define $\gamma(s) := (s,\tau(s))$ for $s \in J$.

By Lemma \ref{l:integral_along_curves} below, we have an integral representation of $\phi$ along $\gamma$:
\begin{equation}\label{form26} \phi(s,\tau(s)) = \int_{0}^s \nabla^{\phi}\phi(r,\tau(r))\;dr, \qquad s \in J. \end{equation}
Hence, for $y \in J$,
\begin{equation}\label{form27} \tau(y) = \int_{0}^{y} \tau'(s) \, ds \stackrel{\eqref{ODEa}}{=} \int_{0}^{y} \phi(s,\tau(s)) \, ds \stackrel{\eqref{form26}}{=} \int_{0}^{y} \int_{0}^{s} \nabla^{\phi}\phi(r,\tau(r)) \, dr \, ds.
\end{equation}
Since $|\nabla^{\phi}\phi| \leq L$ by hypothesis, we see that $y \mapsto \tau(y)$ cannot blow up in finite time; hence $J = \R$, and \eqref{form27} even gives us the quantitative bound
\begin{equation}\label{form28} |\tau(y)| \lesssim_{L} y^{2}, \qquad y \in J = \R. \end{equation}
Now, we estimate the difference $|\phi(y,t) - \nabla^{\phi}\phi(0,0)y|$ as follows:
\begin{equation}\label{form29} |\phi(y,t) - \nabla^{\phi}\phi(0,0)y| \leq |\phi(y,t) - \phi(y,\tau(y))| + |\phi(y,\tau(y)) - \nabla^{\phi}\phi(0,0)y|. \end{equation}
We start with the second term in \eqref{form29}. Using again the integration formula \eqref{form26}, then Definition \ref{d:C1,alpha,W}, and finally \eqref{form28},
\begin{align*} |\phi(y,\tau(y)) - \nabla^{\phi}\phi(0,0)y| & = \int_{0}^{y} |\nabla^{\phi}\phi(s,\tau(s)) - \nabla^{\phi}\phi(0,0)| \, ds\\
& \lesssim_{H} \int_{0}^{y} \|(s,\tau(s))\|^{\alpha} \, ds\\
& \lesssim_{L} \int_{0}^{y} |s|^{\alpha} \, ds \lesssim |y|^{\alpha + 1} \lesssim \|(y,t)\|^{1 + \alpha}, \end{align*}
if $y>0$; the case $y<0$ works analogously with ``$\int_y^0$'' instead of ``$\int_0^y$''.
The estimate above remains valid when $\alpha = 0$: then the $H$-dependent bound is simply replaced by the estimate $|\nabla^{\phi}\phi(s,\tau(s)) - \nabla^{\phi}\phi(0,0)| \leq 2L$.

To prove the corresponding bound for the first term in \eqref{form29}, we use Proposition \ref{p:improvedRegularity} below:
\begin{displaymath} |\phi(y,t) - \phi(y,\tau(y))| \lesssim_{(\ast)} |t - \tau(y)|^{\tfrac{1}{2}+\tfrac{\alpha}{4}} \lesssim |t|^{\tfrac{1}{2}+\tfrac{\alpha}{4}} + |\tau(y)|^{\tfrac{1}{2}+\tfrac{\alpha}{4}} \lesssim_{L} \|(y,t)\|^{1 + \tfrac{\alpha}{2}},\end{displaymath}
using \eqref{form28} in the final estimate. The implicit constant
in $\lesssim_{(\ast)}$ depends on $L$ if $\alpha = 0$ and $H$ if
$\alpha > 0$. The proof of the proposition is complete.
\end{proof}

The proof of the next proposition is similar to that of Lemma 3.3 in \cite{BSC}, but the setting is slightly different, so we give all the details.

\begin{proposition}\label{p:improvedRegularity}
Fix $\alpha \in [0,1]$. If $\alpha = 0$, assume that $\phi \in C^{1}(\W)$
and $\|\nabla^{\phi}\phi\|_{L^{\infty}(\W)} =: H < \infty$. If $\alpha > 0$, assume instead that $\phi \in C^{1,\alpha}(\W)$ with constant $H$, and additionally $\|\nabla^{\phi}\phi\|_{L^{\infty}(\mathbb{W})} < \infty$. Then the restriction of $\phi$ to any vertical line on the plane $\W$ satisfies
\begin{displaymath} |\phi(y_{0},t_{1}) - \phi(y_{0},t_{2})| \leq 4H^{1/2} |t_{1} - t_{2}|^{\tfrac{1}{2}+\tfrac{\alpha}{4}},
 \qquad y_{0} \in \R, \: t_{1},t_{2} \in \R. \end{displaymath}
\end{proposition}

\begin{remark} Recall that the case $\alpha = 0$ of the proposition above is needed to justify Remark \ref{r:ab}(b), which states that the conditions $\phi \in C^{1}(\W)$ and $\nabla^{\phi}\phi \in L^{\infty}(\W)$ imply that $\phi$ is intrinsic Lipschitz. If the conclusion was known \emph{a priori}, then the proof of the proposition (case $\alpha = 0$) could be significantly simplified. Indeed, assuming that $\phi \colon \W \to \V$ is intrinsic $L$-Lipschitz, then
\begin{displaymath} |\phi(y_{0},t_{1}) - \phi(y_{0},t_{2})| \leq L\|\pi_{\W}(\Phi(y_{0},t_{2})^{-1} \cdot \Phi(y_{0},t_{1}))\| = L|t_{1} - t_{2}|^{\tfrac{1}{2}} \end{displaymath}
for $y_{0},t_{1},t_{2} \in \R$.

In the case $\alpha > 0$, we do not know if the qualitative hypothesis $\|\nabla^{\phi}\phi\|_{L^{\infty}(\mathbb{W})} < \infty$ is necessary, but it is for free in our application, and it simplifies the proof noticeably.
\end{remark}

\begin{proof}[Proof of Proposition \ref{p:improvedRegularity}] We make a counter-assumption:
there exists $y_{0} \in \R$ and $t_{1} < t_{2}$ such that
\begin{equation}\label{form7} \phi(y_{0},t_{1}) - \phi(y_{0},t_{2}) = A(t_{2} - t_{1})^{\tfrac{1}{2} + \tfrac{\alpha}{4}} \end{equation}
for some constant $A >  4H^{1/2}$. Other cases (e.g. with $t_{2} <
t_{1}$) can be treated similarly.
 Also, it suffices to prove the proposition in the case where $H\neq 0$.
 Now, the plan is the following: we will define two curves $\gamma_{1},\gamma_{2} \colon \R \to \W$,
 starting from $(y_{0},t_{1})$ and $(y_{0},t_{2})$, respectively. Relying on the counter assumption \eqref{form7}, we will show that
\begin{itemize}
\item[(a)] There exists $s^{\ast} \in \R$ such that $\gamma_{1}(s^{\ast}) = \gamma_{2}(s^{\ast})$, but
\item[(b)] $\phi(\gamma_{1}(s^{\ast})) \neq \phi(\gamma_{2}(s^{\ast}))$.
\end{itemize}
This contradiction will complete the proof.

We now define the curves $\gamma_{j}$. In fact, we define a curve $\gamma_{(y_{0},t)}$ starting from any point $(y_{0},t)$, $t \in \R$, and then set $\gamma_{j} := \gamma_{(y_{0},t_{j})}$. Fix $t \in \R$, and consider the ODE already familiar from the proof of Proposition \ref{p:affineApproximation} (Proposition \ref{p:affineApproximation2}):
\begin{equation}\label{ODE} \begin{cases} \partial_{s} \tau(s,t) = \phi(s,\tau(s,t)), \\ \tau(y_{0},t) = t. \end{cases} \end{equation}
Our qualitative hypothesis $\|\nabla^{\phi}\phi\|_{L^{\infty}(\mathbb{W})} < \infty$ ensures that the solutions $s \mapsto \tau(s,t)$ exist for all $s \in \R$, recall the argument above \eqref{form28}. Define $\gamma_{(y_{0},t)}(s) := (s,\tau(s,t))$ for $s \in \R$. Note that $\gamma_{(y_{0},t)}$ is a $\calC^{1}$ integral curve of the vector field
\begin{displaymath} \nabla^{\phi} := \partial_{y} + \phi \cdot \partial_{t}, \end{displaymath}
since
\begin{displaymath} \partial_{s} \gamma_{(y_{0},t)}(s) = (1,\partial_{s}\tau(s,t)) \stackrel{\eqref{ODE}}{=} (1,\phi(s,\tau(s,t))) = (1,\phi(\gamma_{(y_{0},t)}(s))) = \nabla^{\phi}_{\gamma_{(y_{0},t)}(s)}. \end{displaymath}
Let
\begin{equation}\label{eq:TildeJ}
J_{0} := [y_0,y_0+B(t_2-t_1)^{\tfrac{1}{2}-\tfrac{\alpha}{4}}]\quad\text{with
}B=\tfrac{A}{8H},
\end{equation}
and
\begin{align*} s^{\ast} := \sup\{s \in J_{0} & : \tau(\sigma,t_{1}) - \tau(\sigma,t_{2}) < 0 \text{ for } \sigma \in [y_{0},s] \\
& \text{ and } \partial_{\sigma} \tau(\sigma,t_{1}) - \partial_{\sigma}\tau(\sigma,t_{2}) > 0 \text{ for } \sigma \in [y_{0},s]\}. \end{align*}
Note that $J^{\ast} := [y_{0},s^{\ast})$ is a non-degenerate interval, since (recalling \eqref{ODE} and \eqref{form7})
\begin{displaymath} \tau(y_{0},t_{1}) - \tau(y_{0},t_{2}) = t_{1} - t_{2} < 0 \quad \text{and} \quad \partial_{s}\tau(y_{0},t_{1}) - \partial_{s}\tau(y_{0},t_{2}) = A(t_{2} - t_{1})^{\tfrac{1}{2} + \tfrac{\alpha}{4}} > 0, \end{displaymath}
and the maps involved are continuous.

The main work goes into establishing the following claim, which bootstraps the qualitative positivity of the derivative $\partial_{s} \tau(s,t_{1}) - \partial_{s}\tau(s,t_{2})$ into a quantitative statement:

\begin{claim}\label{claim1} For all $s \in J^{\ast}$:
\begin{equation}\label{eq:LowerBoundDeriv}
\partial_s\tau(s,t_1)-\partial_s \tau(s,t_2)\geq \tfrac{A}{2}
(t_2-t_1)^{\tfrac{1}{2}+\tfrac{\alpha}{4}}.
\end{equation}
\end{claim}

\begin{proof}[Proof of Claim \ref{claim1}] Since $\tau(y_0,t_1)-\tau(y_0,t_2)=t_1-t_2<0$, and the $s$-derivative of the difference is positive on $J^{\ast}$ by definition, we deduce that
\begin{equation}\label{eq:tauEstimateOnJ}
|\tau(s,t_2)-\tau(s,t_1)|=\tau(s,t_2)-\tau(s,t_1)\leq
\tau(y_0,t_2)-\tau(y_0,t_1)= t_2-t_1,\quad s\in
J^{\ast}. \end{equation}
To proceed, we also record the formula
\begin{equation}\label{form8} \phi(s,\tau(s,t))) = \int_{y_{0}}^{s} \nabla^{\phi}\phi(r,\tau(r,t)) \, dr + \phi(y_{0},t), \qquad s \in \mathbb{R},
\end{equation}
given by Lemma \ref{l:integral_along_curves} below. In particular, for $t \in \{t_{1},t_{2}\}$, this formula is applicable for $s \in J^{\ast}$. We then make the following estimate for $s \in J^{\ast}$, using first \eqref{ODE}, then \eqref{form8}:
\begin{align}\label{form10} |[\partial_{s}\tau(s,t_{1}) - \partial_{s}\tau(s,t_{2})] & - A(t_{2} - t_{1})^{\tfrac{1}{2} + \tfrac{\alpha}{4}}|\\
& = |[\partial_{s}\tau(s,t_{1}) - \partial_{s}\tau(s,t_{2})] - [\partial_{s}\tau(y_{0},t_{1}) - \partial_{s}\tau(y_{0},t_{2})]| \notag \\
& = |[\phi(s,\tau(s,t_{1})) - \phi(s,\tau(s,t_{2}))] - [\phi(y_{0},t_{1}) - \phi(y_{0},t_{2})]| \notag \\
&\label{form47} \leq \int_{y_{0}}^{s}
|\nabla^{\phi}\phi(r,\tau(r,t_{1})) -
\nabla^{\phi}\phi(r,\tau(r,t_{2}))| \, dr. \notag\end{align}
From this point on, the cases $\alpha = 0$ and $\alpha > 0$ require slightly different estimates. If $\alpha = 0$, the expression above is simply bounded by
\begin{equation}\label{form46} \eqref{form10} \leq 2H|s - y_{0}|
\overset{\eqref{eq:TildeJ}}{\leq} 2H B(t_2-t_1)^{\tfrac{1}{2}}.
 \end{equation}
For the case $\alpha > 0$, we first use the $C^{1,\alpha}(\W)$
hypothesis, see formula \eqref{form25}, then
\eqref{eq:tauEstimateOnJ}, and finally the upper bound on the
length of $J^{\ast} \subset J_{0}$, to obtain for all $s\in
J^{\ast}$, $s>y_0$, that
\begin{align}\label{form9} \eqref{form10} \leq 2H\int_{y_{0}}^{s} |\tau(r,t_{1}) - \tau(r,t_{2})|^{\tfrac{\alpha}{2}} \,
dr& \overset{\eqref{eq:tauEstimateOnJ}}{\leq} 2H\int_{y_{0}}^{s}
|t_2-t_1|^{\tfrac{\alpha}{2}} \, dr \notag \\&\leq 2H
|t_2-t_1|^{\tfrac{\alpha}{2}}|s-y_0|\notag
\\&\overset{\eqref{eq:TildeJ}}{\leq} 2H B
|t_2-t_1|^{\tfrac{1}{2}+\tfrac{\alpha}{4}}.
\end{align}
Combining \eqref{form46} and \eqref{form9}, and using $HB = A/8 < A/4$ (recall \eqref{eq:TildeJ}), we
obtain for all $\alpha\in [0,1]$ and $s\in J^{\ast}$ that
\begin{equation}\label{eq:form10}
\eqref{form10} < \tfrac{A}{2}(t_{2} - t_{1})^{\tfrac{1}{2} +
\tfrac{\alpha}{4}}.
\end{equation}
Therefore,
\begin{displaymath}
\partial_s\tau(s,t_1)-\partial_s \tau(s,t_2)\geq \tfrac{A}{2}
(t_2-t_1)^{\tfrac{1}{2}+\tfrac{\alpha}{4}} ,\quad s\in J^{\ast}.
\end{displaymath}
verifying \eqref{eq:LowerBoundDeriv}, and Claim \ref{claim1}. \end{proof}

Now, recall that $\tau(y_{0},t_{1}) - \tau(y_{0},t_{2}) = t_{1} - t_{2}$. Since $s \mapsto \tau(s,t_{1}) - \tau(s,t_{2}) < 0$ for $s \in J^{\ast}$, we may deduce from the derivative lower bound \eqref{eq:LowerBoundDeriv} that
\begin{displaymath} |J^{\ast}| \leq \frac{t_{2} - t_{1}}{\tfrac{A}{2}(t_{2} - t_{1})^{\tfrac{1}{2} + \tfrac{\alpha}{4}}} = \tfrac{2}{A}(t_{2} - t_{1})^{\tfrac{1}{2} - \tfrac{\alpha}{4}} < B(t_{2} - t_{1})^{\tfrac{1}{2} - \tfrac{\alpha}{4}}, \end{displaymath}
using finally that $B = A/(8H)$ and $A > 4H^{1/2}$. In particular $s^{\ast} \in \mathrm{int\,} J_{0}$.

We now claim that $\tau(s^{\ast},t_{1}) = \tau(s^{\ast},t_{2})$. Indeed, by the definition of $J^{\ast}$ and $s^{\ast}$, the continuity of the relevant functions, and since $s^{\ast} \in \mathrm{int\,} J_{0}$, either
\begin{displaymath} \tau(s^{\ast},t_{1}) = \tau(s^{\ast},t_{2}) \quad \text{or} \quad \partial_{s}\tau(s^{\ast},t_{1}) = \partial_{s}\tau(s^{\ast},t_{2}). \end{displaymath}
However, the second possibility is ruled out by the quantitative estimate \eqref{eq:LowerBoundDeriv} and the continuity of $s \mapsto \partial_{s}\tau(s,t_{1}) - \partial_{s}\tau(s,t_{2})$. Thus, $\tau(s^{\ast},t_{1}) = \tau(s^{\ast},t_{2})$, and consequently also
\begin{displaymath} \partial_{s}\tau(s^{\ast},t_{1}) \stackrel{\eqref{ODE}}{=} \phi(s^{\ast},\tau(s^{\ast},t_{1})) = \phi(s^{\ast},\tau(s^{\ast},t_{2})) \stackrel{\eqref{ODE}}{=} \partial_{s}\tau(s^{\ast},t_{2}). \end{displaymath}
But we just argued this is ruled out by \eqref{eq:LowerBoundDeriv}. A contradiction has been reached, and the proof of Proposition \ref{p:improvedRegularity} is complete. \end{proof}

We recall a useful integration formula for intrinsic gradients, which has been used in the proofs of the previous propositions.

\begin{lemma}\label{l:integral_along_curves}
Assume that $\phi \in C^{1}(\W)$. Then, for all $y_0,t\in\mathbb{R}$ and every $\mathcal{C}^1$ function $\tau(\cdot,t):(y_0-\delta,y_0+ \delta)\to \mathbb{R}$ satisfying
\begin{displaymath}
\left\{\begin{array}{ll}\partial_s\tau(s,t)=\phi(s,\tau(s,t)),&s\in (y_0-\delta,y_0+\delta),\\\tau(y_0,t)=t,\end{array} \right.
\end{displaymath}
one has
\begin{displaymath}
\phi(s,\tau(s,t)) = \int_{y_0}^s \nabla^{\phi}\phi(r,\tau(r,t))\;dr + \phi(y_0,t),\quad s\in  (y_0-\delta,y_0+\delta).
\end{displaymath}
\end{lemma}

\begin{proof} Fix $y_{0},t \in \R$, and let $\tau(\cdot,t)$ be the $\calC^{1}$ function from the hypothesis. By assumption, $\phi$ is intrinsic differentiable and $\nabla^{\phi}\phi$ is continuous. Hence, $\phi$ satisfies condition (iii) of Theorem 4.95 in \cite{MR3587666}. It follows that $\phi$ also satisfies the equivalent condition (ii), which in our terminology means that the function $s \mapsto \phi(s,\tau(s,t)) =: f(s)$ is in $\calC^{1}$, and hence can be recovered by integrating its derivative:
\begin{displaymath} \phi(s,\tau(s,t)) = \int_{y_0}^s f'(r) \;dr + \phi(y_0,t),\quad s\in  (y_0-\delta,y_0+\delta). \end{displaymath}
Finally, Theorem 4.95 in \cite{MR3587666} also states that $f'(r) = \nabla^{\phi}\phi(r,\tau(r,t))$, and the proof is complete. \end{proof}

We conclude this section with an example related to the definition of $C^{1,\alpha}(\W)$ functions.
The example demonstrates that the property of having a locally $\alpha$-H\"older continuous intrinsic gradient in the metric space $(\W,d)$ is not left-invariant, unlike the $C^{1,\alpha}(\W)$ definition (Lemma \ref{l:translation}).

\begin{ex}\label{ex1} Fix $\alpha \in (0,1]$, let $\W$ be the $(y,t)$-plane, and consider the function $\phi(y,t) = 1 + |t|^{1 + \tfrac{\alpha}{2}}$. Then, using the formula $\nabla^{\phi}\phi = \phi_{y} + \phi\phi_{t}$, we compute
\begin{align*} \nabla^{\phi}\phi(y,t) & = \sgn(t)(1 + |t|^{1 + \tfrac{\alpha}{2}})(1 + \tfrac{\alpha}{2})|t|^{\tfrac{\alpha}{2}}\\
& = \sgn(t)(1 + \tfrac{\alpha}{2})(|t|^{\tfrac{\alpha}{2}} + |t|^{1 + \alpha}). \end{align*}
We first note that $\nabla^{\phi}(y,t)$ is metrically locally $\alpha$-H\"older continuous. Indeed, fix $(y_{1},t_{1}),(y_{2},t_{2}) \in \W$ with $0 \leq t_{1} \leq t_{2}$, and note that
\begin{displaymath} |\nabla^{\phi}\phi(y_{2},t_{2}) - \nabla^{\phi}\phi(y_{1},t_{1})| \sim (t_{2}^{1 + \alpha} - t_{1}^{1 + \alpha}) + (t_{2}^{\tfrac{\alpha}{2}} - t_{1}^{\tfrac{\alpha}{2}}). \end{displaymath}
For $t_{2} - t_{1} \leq 1$, both terms can be bounded by $(t_{2} - t_{1})^{\alpha/2}$. This, and similar computations in the cases $t_{1} \leq t_{2} \leq 0$ and $t_{1} \leq 0 \leq t_{2}$, imply that $\nabla^{\phi}\phi$ is locally $\alpha$-H\"older continuous in the space $(\W,d)$. Now, consider $\phi^{(p^{-1})}$ with $p = \Phi(0)$, and note that, using \eqref{form25},
\begin{align*} |\nabla^{\phi^{(p^{-1})}} \phi^{(p^{-1})}(y,0) - \nabla^{\phi^{(p^{-1})}}\phi^{(p^{-1})}(0,0)| & = |\nabla^{\phi}\phi(y,\phi(0,0)y) - \nabla^{\phi}\phi(0,0)|\\
& = (1 + \tfrac{\alpha}{2})(|y|^{\tfrac{\alpha}{2}} + |y|^{1 + \alpha}). \end{align*}
The last expression is not bounded by a constant times $\|(y,0)\|^{\alpha}$ as $y \to 0$, so the intrinsic gradient of $\phi^{(p^{-1})}$ is not locally $\alpha$-H\"older continuous at the origin. \end{ex}

\section{Boundedness of the Riesz transform, and removability}\label{sec:rem}

In this section we prove Theorem \ref{mainRem} and its corollaries. We start by recalling the statement of Theorem \ref{mainRem}:
\begin{thm}\label{mainRem2} Assume that $\mu$ is a Radon measure on $\He$, satisfying $\mu(\He)>0$, $\mu(B(p,r)) \leq Cr^{3}$ for $p\in\He,r>0$, and such that the support $\spt \mu$ has locally finite $3$-dimensional Hausdorff measure. If $\calR_{\He}$ is bounded on $L^{2}(\mu)$, then $\spt \mu$ is not removable for Lipschitz harmonic functions. \end{thm}

We will need a few auxiliary results in the proof. The first one establishes that $\He$ is ``horizontally polygonally quasiconvex'', which means that the space is well connected by segments of horizontal lines. A \emph{horizontal line} is simply a set of the form $\pi_{\mathbb{W}}^{-1}\{w\}$ for some vertical subgroup $\W$ and a point $w\in \W$. In other words, the horizontal lines are left cosets of $1$-dimensional horizontal subgroups.

\begin{lemma}\label{l:horizontalQuasiconvexity} Let $z_{1},z_{2} \in \He$. Then, there exist five horizontal segments $\ell_{1},\ldots,\ell_{5}$ with connected union, containing $z_{1},z_{2}$, and such that $\sum \calH^{1}(\ell_{j}) \leq 3d(z_{1},z_{2})$.
\end{lemma}

\begin{proof} We may assume that $z_{1} = \mathbf{0}$, since left translations are isometries and send horizontal lines to horizontal lines. Now, we first discuss the case $z_{1} = \mathbf{0}$ and $z_{2} = (0,0,t_{2})$. Plot a square of side-length $\sqrt{|t|}$ in the $(y,t)$-plane, with one vertex at $z_{1}$. Then, it is well-known (see Sections 2.2--2.3 in \cite{MR2312336}) that there exists a piecewise linear horizontal curve (a lift of the square), consisting of four horizontal line segments, connecting $z_{1}$ to $(0,0,t_{2})$, and with total length $4\sqrt{|t_{2}|} = 2d(z_{1},z_{2})$. This completes the case $z_{2} = (0,0,t_{2})$.

If $z_{2} = (x_{2},y_{2},t_{2})$ is arbitrary, first connect $z_{1} = \mathbf{0}$ to the point $(x_{2},y_{2},0)$ with a horizontal line of length $\leq d(z_{1},z_{2})$. Then, by the previous discussion (and a left-translation), the points $(x_{2},y_{2},0)$ and $z_{2}$ can be connected by four horizontal line segments of total length $\leq 4\sqrt{|t_{2}|} \leq 2d(z_{1},z_{2})$.
\end{proof}

 The following result is an analogue of Lemma 7.7 in \cite{zbMATH01249699} for vertical Heisenberg projections. It was a surprise to the authors that the proof works even though the maps $\pi_{\W}$ are not Lipschitz.
 \begin{lemma}\label{l:preImage} Fix a vertical plane $\W$, and let $E \subset \He$. Then,
 \begin{displaymath} \int^{\ast}_{\W} \calH^{s - 3}(E \cap \pi_{\W}^{-1}\{w\}) \, d\calL^{2}(w) \lesssim \calH^{s}(E), \quad 3 \leq s \leq 4. \end{displaymath}
 \end{lemma}

 \begin{proof} If $\calH^{s}(E) = \infty$, there is nothing to prove. Otherwise, for all $k \in \N$, and cover $E$ by balls $B_{i,k} := B(z_{i,k},r_{i,k})$, $i,k \in \N$, with $0 < r_{i,k} < \epsilon$ and
 \begin{displaymath} \sum_{i \in \N} r_{i,k}^{s} \lesssim \calH_{1/k}^{s}(E) + \tfrac{1}{k}. \end{displaymath}
 Then, by Lemma 2.20 in \cite{MR3511465}, we have
 \begin{displaymath} \calL^{2}(\pi_{\W}(B_{i,k})) = Cr_{i,k}^{3}, \qquad i,k \in \N, \end{displaymath}
 for some positive and finite constant $C$. Consequently, by Fatou's lemma
 \begin{align*} \int_{\W}^{\ast} \calH^{s - 3}(E \cap \pi_{\W}^{-1}\{w\}) \, d\calL^{2}(w) & \leq \liminf_{k \to \infty} \int_{\W}^{\ast} \sum_{i \in \N} \diam(E \cap B_{i,k} \cap \pi_{\W}^{-1}\{w\})^{s - 3} \, d\calL^{2}(w)\\
 & \leq \liminf_{k \to \infty} \sum_{i \in \N} \int_{\pi_{\W}(B_{i,k})} \diam(E \cap B_{i,k} \cap \pi^{-1}_{\W}\{w\})^{s - 3} \, d\calL^{2}(w)\\
 & \lesssim \liminf_{k \to \infty} \sum_{i \in \N} \calL^{2}(\pi_{\W}(B_{i,k}))r_{i,k}^{s - 3} \lesssim \liminf_{k \to \infty} \sum_{i \in \N} r_{i,k}^{s} \lesssim \calH^{s}(E),  \end{align*}
 as claimed. \end{proof}

In this paper, Lemmas \ref{l:horizontalQuasiconvexity} and \ref{l:preImage} are only needed to prove the following criterion for a continuous map $f \colon \He \to \R$ to be Lipschitz:
\begin{lemma}\label{l:lipschitzLemma} Let $E \subset \He$ be a set of locally finite $3$-dimensional measure and let $f \colon \He \to \R$ be continuous. If $f \in \calC^{1}(\He \setminus E)$ and $\nabla_{\He} f \in L^{\infty}(\He \setminus E)$, then $f$ is Lipschitz on $\He$. \end{lemma}

\begin{proof} Fix $z_{1},z_{2} \in \He$, and pick a bounded open set $U$ containing both $z_{1},z_{2}$, and also the line segments $\ell_{1},\ldots,\ell_{5}$ constructed in Lemma \ref{l:horizontalQuasiconvexity}. For notational convenience we rename $z_{2}$ as $z_{6}$: then we can say that $\ell_{j}$ connects $z_{j}$ to $z_{j + 1}$ for all $1 \leq j \leq 5$, where $z_{j},z_{j + 1}$ are the endpoints of $\ell_{j}$.

Fix $\epsilon > 0$ small. We would like to replace the segments $\ell_{j}$ by segments $\tilde{\ell}_{j}$ with the following properties:
\begin{itemize}
\item[(i)] For $1 \leq j \leq 5$, the segment $\tilde{\ell_{j}}$ connects $z_{j,2}$ to $z_{j + 1,1}$, where $z_{j,1},z_{j,2} \in B(z_{j},\epsilon)$,
\item[(ii)] $\calH^{0}(E \cap \tilde{\ell}_{j}) < \infty$.
\item[(iii)] $\calH^{1}(\tilde{\ell}_{j}) \sim \calH^{1}(\ell_{j})$.
\end{itemize}
We cannot guarantee that the union of the segments $\tilde{\ell}_{j}$ is connected, but this will not be an issue, if $\epsilon > 0$ is chosen small enough. Note that $\ell_{j}$ is contained on a line of the form $\pi_{\W}^{-1}\{w\}$ for a certain vertical plane $\W$, and $w \in \W$. Now, we pick $\tilde{w} \in \W$ very close to $w$ such that $\calH^{0}(E \cap \pi_{\W}^{-1}\{\tilde{w}\}) < \infty$, using Lemma \ref{l:preImage}. Then, a suitable line segment $\tilde{\ell}_{j} \subset \pi_{\W}^{-1}\{\tilde{w}\}$ satisfies (i)--(iii).

Now, we can finish the proof of the lemma. Note that $f$ restricted to any line segment $\tilde{\ell}_{j}$ is Lipschitz with constant depending only on the $L^{\infty}$-norm of $\nabla_{\He}f$, as one can see by piecewise integration. Hence,
\begin{displaymath} |f(z_{j + 1,1}) - f(z_{j,2})| \lesssim d(z_{j + 1,1},z_{j,2}) = \calH^{1}(\tilde{\ell}_{j}) \sim \calH^{1}(\ell_{j}). \end{displaymath}
Now, it follows from the continuity of $f$, the triangle inequality, and the choice of the segments $\ell_{j}$, that
\begin{displaymath} |f(z_{2}) - f(z_{1})| \lesssim \sum_{j = 1}^{5} \calH^{1}(\ell_{j}) + \eta(\epsilon) \sim d(z_{1},z_{2}) + \eta(\epsilon), \end{displaymath}
where
\begin{displaymath} \eta(\epsilon) = \max \{|f(z_{j}) - f(z_{j}')| : 1 \leq j \leq 6 \text{ and } z_{j}' \in B(z_{j},\epsilon)\}. \end{displaymath}
Since $\eta(\epsilon) \to 0$ as $\epsilon \to 0$ by the continuity of $f$, the proof is complete. \end{proof}

The following theorem provides a dualisation of weak $(1,1)$ inequalities, and is due to Davie and \O{}ksendal \cite{DO}. The proof can also be found in \cite[Lemma 4.2]{MR1372240} or \cite[Theorem 4.6]{tolsabook}. Let $X$ be a locally compact Hausdorff space; the reader may think that $X = \He$. We recall that $\mathcal{C}_0(X)$ denotes the vector space of continuous functions which vanish at infinity: $f \in \mathcal{C}_0(X)$ if and only if for every $\varepsilon>0$ there exists a compact set $E \subset X$ such that $|f(x)|< \varepsilon$ for all $x \in X \setminus E$. It is well known that $\mathcal{C}_0(X)$ equipped with the sup-norm is a Banach space. The dual of  $\mathcal{C}_0(X)$ is the Banach space $(\mathcal{M}(X),\|\cdot\|)$; the space of all signed Radon measures on $X$ with finite total variation $\|\nu\|$.

\begin{thm}\label{dual} Let $T_{i} \colon \mathcal{M}(X) \ra \mathcal{C}_0(X)$, $i \in \{1,2\}$, be bounded linear operators and let $T_{i}^*:\mathcal{M}(X) \ra \mathcal{C}_0(X)$ be the (formal) adjoint operators of $T_{i}$ satisfying,
\begin{equation}
\label{adjmeas}
\int (T_{i} \nu_1) \, d\nu_2=\int (T_{i}^\ast \nu_2) \, d \nu_1,
\end{equation}
for all $\nu_1,\nu_2 \in \mathcal{M}(X)$. Assume also that $\mu$ is a Radon measure on $X$ such that the operators $T_{i}^\ast : \mathcal{M}(X) \ra L^{1,\infty}(\mu)$ are bounded, that is
\begin{equation}
\label{weak11}
\mu (\{x \in X: |T_{i}^\ast\nu (x)|> \lambda\}) \leq C \frac{\|\nu\|}{\lambda}
\end{equation}
for all $\nu \in \mathcal{M}(X)$ and $\lambda>0$. Then, for every Borel set $E \subset X$ with $0 < \mu(E)<\infty$ there exists a function $h \in L^{\infty}(\mu)$, satisfying $0 \leq h(x) \leq \chi_{E}(x)$ for $\mu$ almost every $x \in X$, such that
$$\mu(E) \leq 2 \int h \,d\mu \quad\text{and}\quad \|T_i(h \mu)\| \leq 3 C.$$
\end{thm}

Recall that if $D \subset \He$ is open, a function $f: D \ra \R$ is called \emph{harmonic} if it is a distributional solution to the sub-Laplacian equation $\Delta_{\He} f=0$, see \eqref{eq:subLapl} for the definition of $\Delta_{\He}$. We record that by H\"ormander's theorem, see for instance \cite[Theorem 1]{BLU}, all distributional solutions to the sub-Laplacian equation are in $\calC^\infty(\He)$. Hence, in accordance to the Euclidean case, one can naturally define removable sets for Lipschitz harmonic functions in $\He$; these were introduced in \cite{CM}.

\begin{df}
\label{d:rem} A closed set $E \subset \He$ is called {\em removable for Lipschitz harmonic functions}, if whenever  $D \supset E$ is open, every  Lipschitz function $f:D \ra \R$ which is  harmonic in $D\setminus E$ is also harmonic in $D$. 
\end{df}



We are now ready to prove Theorem \ref{mainRem2}:

\begin{proof} We will follow a well known scheme which dates back to Uy \cite{uy}. See also \cite[Theorem 4.4]{MR1372240}, \cite[Chapter 4]{tolsabook} and \cite[Chapter 2]{zbMATH01249699}. Recall from Section \ref{ss:KernelsAndSIO} that the truncated $\He$-Riesz transform $\mathcal{R}$ of a Radon measure $\nu$ at level $\varepsilon$ consists of two components:
$$\mathcal{R}^i_{\varepsilon} \nu (p)=\int_{\|q^{-1}\cdot p\|>\varepsilon} \mathcal{K}^i(q^{-1}\cdot p) d \nu (q), \quad p \in \He,\; i=1,2,$$
where $\mathcal{K}(p)=(\mathcal{K}^1(p), \mathcal{K}^2(p))=\nabla_{\He} \|p\|^{-2}$. See \eqref{KExplicit} for an explicit formula for $\mathcal{K}^i$.

%

We fix a compact set $E \subset \spt \mu$ with $0<\mu(E)<\infty$. Our goal is to apply Theorem \ref{dual}, and the first step is to define smoothened versions of the operators $\mathcal{R}^i_{\mu,\varepsilon},i=1,2,$ as in Section \ref{sec:vert1}. Let $\phi:\R \ra [0,1]$ be a $\calC^{\infty}$ function such that $\phi=0$ on $(-1/2,1/2)$ and $\phi=1$ on $\R \setminus (-1,1)$. We then define
\begin{displaymath}
\tilde{\mathcal{R}}^i_{\mu,\varepsilon}(f)(p)=\int \phi \left( \frac{\|q^{-1}\cdot p\|}{\varepsilon}\right)\mathcal{K}^i(q^{-1}\cdot p)f(q) \, d\mu(q),\quad i=1,2,
\end{displaymath}
and
$$\tilde{\mathcal{R}}^i_{\varepsilon}\nu(p)=\int \phi \left( \frac{\|q^{-1}\cdot p\|}{\varepsilon}\right)\mathcal{K}^i(q^{-1}\cdot p) \, d\nu(q),\quad i=1,2,$$
for $\nu \in \mathcal{M}(\He)$. The operator $\tilde{\mathcal{R}}^i_{\varepsilon}$, and its formal adjoint $\tilde{\mathcal{R}}^{i,\ast}_{\varepsilon}$, both map $\mathcal{M}(\He)$ to $\mathcal{C}_0(\He)$, something that is not true for $\mathcal{R}^i_{\varepsilon}$. The operators $\mathcal{R}^i_{\mu,\varepsilon},i=1,2,$ are uniformly bounded in $L^2(\mu)$ by hypothesis, and by Remark \ref{r:adjoint} the same holds for their adjoints. Arguing as in Lemma \ref{l:smoothcomp}, and using the $L^2(\mu)$-boundedness of the maximal function $M_{\mu}$ we deduce that the smoothened operators $\tilde{\mathcal{R}}^i_{\mu,\varepsilon}$ and their adjoints $\tilde{\mathcal{R}}^{i,\ast}_{\mu,\varepsilon}$ are also uniformly bounded in $L^2(\mu)$: there exists a constant $C_1>0$ such that for all $f \in L^2(\mu)$ and every $\ve>0$,
\begin{equation}
\label{adjsmoothl2}\|\tilde{\mathcal{R}}^{i,\ast}_{\mu,\varepsilon}( f)\|_{L^2(\mu)} \leq C_1 \|f\|_{L^2(\mu)}, \quad i=1,2.
\end{equation}
 By a result of Nazarov, Treil and Volberg \cite[Corollary 9.2]{MR1626935} the uniform $L^2$-boundedness \eqref{adjsmoothl2} implies that the operators $\tilde{\mathcal{R}}^{i,\ast}_{\ve}$ map $\mathcal{M}(\He)$ to $L^{1,\infty}(\mu)$  boundedly. That is, for every $\lambda>0$ and every $\nu \in \mathcal{M}(\He)$,
\begin{equation}
\label{weak11}
\mu (\{p \in \He: |\tilde{\mathcal{R}}^{i,\ast}_{\ve}\nu (p)|> \lambda\}) \leq C_2 \frac{\|\nu\|}{\lambda},
\end{equation}
where $C_2$ only depends on $C_1$. We can now apply Theorem \ref{dual} to the operators $\tilde{\mathcal{R}}^i_{\varepsilon}$. We thus obtain, for every $\ve>0$, a function $h_{\ve} \in L^{\infty}(\mu)$ such that $0 \leq h_{\ve}(p) \leq \chi_{E}(p)$ for $\mu$ almost every $p \in \He$, and
\begin{equation}
\label{hemeas}
\int h_{\ve} \, d\mu \geq \mu(E)/2
\end{equation}
and
\begin{equation}
\label{linfbound}
\|\tilde{\mathcal{R}}^i_{\mu,\varepsilon} h_{\ve}\| \leq 3 C_2.
\end{equation}
The norm appearing in \eqref{linfbound} is the $\sup$-norm in $\He$. Let $\Phi(p)=\|p\|^{-2}$ be (a multiple of) the fundamental solution of the sub-Laplacian in $\He$. For $\ve>0$, we consider the functions
$$f_{\ve}(p)=\int \Phi(q^{-1} \cdot p) h_{\ve}(q) \, d\mu(q), \quad p \in \He.$$
Recall that we have $h_\ve(p) =0$ for $\mu$ almost every $p \in \He \stm E$.  Moreover since $E$ is compact and $\mu$ satisfies $\mu(B(p,r)) \lesssim r^{3}$, it follows easily that $f_{\ve}$ is well defined for all $p \in \He$. The left invariance of $\nabla_{\He}$  implies that
$$\nabla_{\He} f_{\ve} (p)=(\tilde{\mathcal{R}}^1_{\mu,\varepsilon} h_{\ve}(p),\tilde{\mathcal{R}}^2_{\mu,\varepsilon} h_{\ve}(p))$$
for $p \in \He \setminus N_{\ve}(E)$ where $N_{\ve}(E)=\{p \in \He: d(p,E)<\ve\}$. By \eqref{linfbound}, we infer that
\begin{equation}
\label{nablabound}
|\nabla_{\He} f_{\ve} (p)| \leq 6 C_2
\end{equation}
for $p \in \He \setminus N_{\ve}(E)$.

As a consequence of the Banach-Alaoglu Theorem, see  \cite[Corollary 3.30]{Brezis}, there exists a sequence $\ve_n \ra 0$ such that $h_{\ve_n}$ converges to some function $h \in L^\infty (E, \mu|_{E})$ in the weak$^\ast$ topology. This means that
\begin{equation}
\label{weakstar}\int_E h_{\ve_n}g\, d\mu \ra \int_E hg\, d\mu,
\end{equation}
for all $g \in L^1(E, \mu|_{E})$.  We now apply \eqref{weakstar} to the function $g=1$ and we invoke \eqref{hemeas} to get
\begin{equation}
\label{lastbound}
\int_E h \, d\mu \geq \mu(E)/2.
\end{equation}
It follows easily, see  \cite[Proposition 3.13]{Brezis}, that $\|h\|_{L^\infty(E, \mu|_{E})} \leq 1$. Hence $0 \leq h(p) \leq \chi_{E}(p)$ for $\mu$ almost every $p \in E$. Assuming that $h=0$ outside $E$ we set $\nu=h \,d\mu$, so  $\spt \nu \subset E$ and
\begin{equation}
\label{3growth}
\nu(B(p,r)) \leq C r^{3}, \quad\mbox{ for }p\in \He, r>0,
\end{equation}
where $C$ is the ADR constant of $\mu$.

Let
$$f(p) :=\int \Phi(q^{-1}\cdot p) d\nu(q), \qquad p \in \He.$$
Our aim is to show that $f$ is a Lipschitz function which is harmonic in $\He \stm E$  but not in $\He$.  First note that
$$\lim_{n \ra \infty} f_{\ve_n}(p)=f(p)$$
for $p \in \He \stm E$. Since $E$ is compact, \eqref{3growth} implies that $f$ is well defined in $\He$. But more is true; it turns out that $f$ is continuous. To see this, let $p \in \He$ and take a sequence $(p_n)_{n\in \N}$ such that $p_n \ra p$. Let $\delta>0$ and fix $r< \delta/ 100 C$, where $C$ is the constant from \eqref{3growth}. We then write
\begin{equation*}
\begin{split}
|f(p_n)-f(p)| &\leq \int_{E\setminus B(p,r)} |\Phi(q^{-1} \cdot p_n) - \Phi(q^{-1} \cdot p)|\,d \nu(q) \\
&\qquad+ \int_{B(p,r)} |\Phi(q^{-1} \cdot p_n) - \Phi(q^{-1} \cdot p)|\, d \nu(q)\\
&=I_1(n)+I_2(n).
\end{split}
\end{equation*}
Since $E$ is compact, the continuity of $\Phi$ and \eqref{3growth} imply that there exists some $n_1 \in \N$ such that $I_1(n)<\delta/100$ for $n \geq n_1$. We now estimate $I_2(n)$. Pick $n_2 \in \N$ such that $d(p_n,p)<r$ for all $n \geq n_2$. Hence, if $d(q,p)<r$, we also have that $d(q,p_n)<2r$. Therefore using \eqref{3growth} and integrating on annuli we get
\begin{equation*}
\begin{split}I_2(n) &\leq \int_{B(p_n,2r)}\frac{1}{\|q^{-1} \cdot p_n\|^2}d \nu(q)+ \int_{B(p,r)}\frac{1}{\|q^{-1} \cdot p\|^2}d \nu(q) \leq 24 C r< \delta.
\end{split}
\end{equation*}
Putting these estimates together we have that $|f(p_n)-f(p)|<\delta$ for $n>n_1+n_2$, thus $f$ is continuous.

Note that the sequence $(f_{\ve_n})$ is equicontinuous
on compact subsets of $\He \stm E$, and by the Arzel\`a-Ascoli theorem
there exists a subsequence $(f_{\ve_{n_l}})$  which converges uniformly
on compact subsets of $\He \stm E$. Therefore the Mean Value Theorem for
sub-Laplacians and its converse, see \cite[Theorem 5.5.4 and Theorem 5.6.3]{BLU},  imply that $f$ is harmonic in $\He \stm E$.

We will now show that $f$ is Lipschitz in $\He$. We intend to apply  Lemma \ref{l:lipschitzLemma}. The first step comes again from H\"ormander's theorem \cite[Theorem 1]{BLU}; since $f$ is harmonic in $\He \stm E$, it is also in $\calC^\infty (\He \setminus E)$. We already proved that $f$ is continuous in $\He$ so we only need to show that $\nabla_{\He} f \in L^{\infty}(\He \setminus E)$. To this end, pick any  $p \in \He \stm E$ and let $r>0$ be such that $B_{cc}(p,r) \subset \He \stm E$, where $B_{cc}(p,r)$ denotes a ball with respect to the standard \emph{sub-Riemannian distance} on $\He$.
For $n$  big enough, \cite[Proposition 3.9]{CM} and \eqref{nablabound} imply that $f_{\ve_n}$ is Lipschitz in $B_{cc}(p,r)$ with $\Lip(f_{\ve_n}|_{B_{cc}(p,r)}) \leq 6 C_2 C_3$, where $C_3$ is the comparison constant of $d$ and $d_{cc}$. We know that $f_{\ve_n} \ra f$ pointwise in (the compact set) $\bar{B}_{cc}(p,r/2)$ hence $f$ is Lipschitz in $\bar{B}_{cc}(p,r/2)$ with $\Lip(f|_{\bar{B}_{cc}(p,r/2)}) \leq 6 C_2 C_3$. Applying \cite[Proposition 3.9]{CM} once more we conclude that $|\nabla_{\He} f (p)| \leq 6 C_2 C_3.$ Therefore we have shown that $\nabla_{\He} f \in L^{\infty}(\He \setminus E)$, and by Lemma \ref{l:lipschitzLemma}, $f$ is Lipschitz.

Finally in order to finish the proof of the theorem it suffices to show that $f$ is {\em not} harmonic in $\He$. By the definition of the fundamental solution, see \cite[Definition 5.3.1 (iii)]{BLU} we deduce that there exists some $c>0$ (which comes from the normalization of the fundamental solution)
$$\langle \Delta_{\He} f, 1\rangle=-c \int h\, d\mu \leq -c \, \mu(E)/2<0.$$
Hence $f$ is not harmonic in $\He$, and the proof of Theorem \ref{mainRem2} is complete.
\end{proof}

Corollary \ref{cor1} is clear: Let $\phi \in C^{1,\alpha}(\W)$ with $\alpha > 0$, and $\Gamma = \Gamma(\phi) \subset \He$. If $E \subset \Gamma$ is closed with $\calH^{3}(E) > 0$, then the  measure $\mu = \calH^{3}|_{E}$ satisfies the hypotheses of Theorem \ref{mainRem2}, by the 3-$AD$ regularity of $\calH^{3}|_{\Gamma}$, Theorem \ref{mainABC}, and the discussion in the end of Section \ref{s:cancel} . Hence $E$ is not removable.

Corollary \ref{cor2} requires a little argument. We recall the statement:

\begin{cor} Let $\alpha > 0$. Assume that $\Omega \subset \R^{2} \cong \W$ is open, and $\phi \in \calC^{1,\alpha}(\Omega)$. If $E$ is a closed subset of the intrinsic graph $\Gamma = \{w \cdot \phi(w) : w \in \Omega\}$ with $\calH^{3}(E) > 0$, then $E$ is not removable.
\end{cor}

\begin{proof} Replacing $E$ by a compact subset with positive $3$-dimensional measure, we may assume that $E$ itself is compact, and so is $\pi_{\W}(E) \subset \Omega$. Let $\psi$ be a compactly supported $\calC^{\infty}(\R^2)$ function with $\chi_{\pi_{\W}(E)} \leq \psi \leq \chi_{\Omega}$. Then
\begin{displaymath}
\widetilde{\phi} := \left\{ \begin{array}{ll}\psi \phi,&\text{ in }\Omega,\\ \psi,&\text{ in }\mathbb{R}^2 \setminus \Omega,\end{array} \right.
\end{displaymath}
defines a  $\calC^{1,\alpha}$ function which agrees with $\phi$ on $\pi_{\W}(E)$.
 (This is where the Euclidean regularity is required; we do not know if \emph{intrinsic} $C^{1,\alpha}$ regularity is preserved under such product operations.) Now, by Remark \ref{e:euclideanHolder}, we infer that $\widetilde{\phi} \in C^{1,\alpha}(\W)$, and $\widetilde{\phi}$ is certainly compactly supported. Consequently, by Corollary \ref{cor1}, compact subsets of the intrinsic graph $\Gamma(\widetilde{\phi})$ with positive $3$-measure are not removable. In particular, this is true for $E \subset \Gamma(\widetilde{\phi})$.  \end{proof}

\bibliographystyle{plain}
\bibliography{references}

\end{document}